\documentclass{amsart}

\usepackage{amssymb}
\usepackage{amsmath}
\usepackage{amsthm}
\usepackage{amsfonts}

\usepackage{a4wide}
\usepackage{longtable}
\usepackage{multirow}
\usepackage{setspace}

\newtheorem{proposition}{Proposition}[section]
\newtheorem{theorem}{Theorem}[section]
\newtheorem{corollary}{Corollary}[section]

\theoremstyle{definition}
\newtheorem{definition}{Definition}[section]

\newtheorem{remark}{Remark}[section]

\newtheorem{example}{Example}[section]

\newtheorem{conjecture}{Conjecture}[section]

\numberwithin{equation}{section}

\begin{document}


\title[Generalized Artin pattern and Scholz's conjecture]{Generalized Artin pattern of heterogeneous multiplets \\
of dihedral fields and proof of Scholz's conjecture}

\author{Daniel C. Mayer}
\address{Naglergasse 53\\8010 Graz\\Austria}
\email{algebraic.number.theory@algebra.at}
\urladdr{http://www.algebra.at}

\thanks{Research supported by the Austrian Science Fund (FWF): projects J 0497-PHY and P 26008-N25}


\date{April 12, 2019}


\begin{abstract}
The concept of Artin transfer pattern \(((\ker(T_{K,N_i}))_i,(\mathrm{Cl}_p(N_i))_i)\)
for homogeneous multiplets \((N_1,\ldots,N_m)\)
of unramified cyclic prime degree \(p\) extensions \(N_i/K\) of a base field \(K\)
with \(p\)-class transfer homomorphisms \(T_{K,N_i}:\,\mathrm{Cl}_p(K)\to\mathrm{Cl}_p(N_i)\)
is generalized for heterogeneous multiplets of ramified extensions.
By application to quadratic subfields \(K\)
of dihedral fields \(N\) of degree \(2p\) with an odd prime \(p\),
a conjecture of Scholz
concerning the index of subfield units, \((U_N:U_0)\), 
for ramified extensions \(N/K\) with conductor \(f>1\)
is verified computationally.
\end{abstract}

\maketitle


\section{Introduction}
\label{s:Intro}
\noindent
The concept of \textit{Artin pattern}
was introduced by ourselves
\cite{Ma2016}
in a new technique
for finding the metabelian Galois group \(\mathrm{G}_p^2(K)=\mathrm{Gal}(\mathrm{F}_p^2(K)/K)\)
of the second Hilbert \(p\)-class field \(\mathrm{F}_p^2(K)\)
of a number field \(K\) with non-cyclic \(p\)-class group
\cite{Ma2012},
for an arbitrary prime \(p\ge 2\).
The strategy of \textit{pattern recognition via Artin transfers}
\cite{Ma2011}
is a method for using number theoretic information
on the \textit{unramified} cyclic extensions \(N/K\) of relative degree \(p\)
in a group theoretic search for \(G=\mathrm{G}_p^2(K)\)
with the aid of the \(p\)-group generation algorithm
\cite{Nm,Ob}.
According to the Artin reciprocity law of class field theory
\cite{Ar1927},
the kernels and targets of the \(p\)-class extension homomorphisms
\(T_{K,N}:\,\mathrm{Cl}_p(K)\to\mathrm{Cl}_p(N)\),
\(\mathfrak{a}\cdot\mathcal{P}_K\mapsto\mathfrak{a}\mathcal{O}_N\cdot\mathcal{P}_N\),
correspond to the kernels and targets of the \textit{transfer homomorphisms}
\(T_{G,H}:\,G/G^\prime\to H/H^\prime\)
from \(G\) to its maximal subgroups \(H<G\) of index \(p\)
\cite{Ar1929,Ma2010}.
In
\cite{Ma2016}
we defined the \textit{Artin pattern} of \(K\) as the collection
\(\mathrm{AP}(K)=(\varkappa(K),\tau(K))\) of kernels 
\(\varkappa(K)=(\ker(T_{K,N}))_N\),
and targets
\(\tau(K)=(\mathrm{Cl}_p(N))_N\) of \(p\)-class extensions,
where \(N\) varies over all
unramified cyclic extensions \(N/K\) of relative degree \(p\).

In the present article, our principal aim is to use \textit{extended Artin patterns}
for a \textit{systematic classification} of
\textit{ramified} cyclic extensions \(N/K\) of odd prime degree \(p\)
over quadratic or cyclotomic fields \(K\)
with the aid of transfer data and principal factorization types,
rather than for finding a suitable automorphism group.

The precise definition of \textit{homogeneous} and \textit{heterogeneous multiplets}
of dihedral fields is given in \S\
\ref{s:Multiplets}
using concepts of the multiplicity theory of dihedral discriminants
\cite{Ma1992,Ma2014}.

After the illustration of two important special situations involving heterogeneous multiplets,
which are crucial for verifying the truth of several hypotheses of Scholz in
\cite{So},
\S\
\ref{s:Pattern}
introduces the \textit{generalized Artin pattern} of a heterogeneous multiplet.
In \S\
\ref{s:ObjectsInvariants}
we distinguish between multiplets of \textit{objects}
and associated multiplets of \textit{invariants}.

For obtaining deeper structural insight,
we sort the components of the Artin pattern
by \textit{differential principal factorization} (DPF) \textit{types}
in \S\
\ref{s:Trichotomy},
where we prove that the \(\mathbb{F}_p\)-vector space \(\mathcal{P}_{N/K}/\mathcal{P}_K\)
of primitive ambiguous principal ideals of a number field extension \(N/K\)
can be endowed with a natural trichotomic direct product structure.

Finally,
as an application of the notions of multiplets, Artin patterns and DPF types,
the Conjecture of Scholz
\cite{So}
is stated, refined, and proved completely in \S\
\ref{s:ScholzConjecture}.


\section{Homogeneous and heterogeneous multiplets of dihedral fields}
\label{s:Multiplets}
\noindent
In the sequel we fix an \textit{odd} prime number \(p\ge 3\).
\begin{definition}
\label{dfn:Multiplets}
\noindent
Let \((N_1,\ldots,N_m)\) be a multiplet of \(m\) pairwise non-isomorphic
dihedral fields of degree \(2p\)
which share a common quadratic subfield \(K<N_i\) \((1\le i\le m)\)
with discriminant \(d_K\).
\begin{enumerate}
\item
The multiplet \((N_1,\ldots,N_m)\) is called \textit{homogeneous} if
\begin{enumerate}
\item
all members share a common class field theoretic conductor \(f=f(N_i/K)\), and
\item
the number \(m=m_p(d_K,f)\) is the \(p\)-multiplicity of \(f\) with respect to \(K\)
\cite{Ma1991,Ma2014}.
\end{enumerate}
\item
The multiplet \((N_1,\ldots,N_m)\) is called \textit{heterogeneous} if
\begin{enumerate}
\item
all members possess conductors \(f_i=f(N_i/K)\)
which are divisors of an assigned \(p\)-admissible conductor \(f_0\) for \(K\)
\cite{Ma2014},
and
\item
the number \(m=\sum_{f\mid f_0}\,m_p(d_K,f)\) is the sum of all \(p\)-multiplicities
of divisors \(f\) of \(f_0\) with respect to \(K\).
\end{enumerate}
\end{enumerate}
\end{definition}


\section{Generalized Artin pattern}
\label{s:Pattern}
\noindent
Since the extended Artin pattern of a composite conductor
has a structure with high complexity,
we abstain from maximal possible generality
and restrict ourselves to the simplest non-trivial situation
where \(f_0=q\in\mathbb{P}\) is a \(p\)-admissible \textit{prime conductor}
for a quadratic field \(K\) with \(p\)-class rank
\(\varrho_p=\mathrm{rank}_p(\mathrm{Cl}(K))\ge 1\),
that is, \(K\) has class number divisible by \(p\).
Then \(f_0\) has only two divisors,
\(f=1\) corresponding to \textit{unramified} abelian extensions \(N_{1,i}/K\), \(1\le i\le m_p(d_K,1)\), and
\(f=q\) corresponding to \textit{ramified} abelian extensions \(N_{q,j}/K\), \(1\le j\le m_p(d_K,q)\).


\begin{example}
\label{exm:Multiplet}
\noindent
Before we give an abstract definition of the extended Artin pattern,
we put \(p=3\) the smallest odd prime,
and we start with two illuminating examples for heterogeneous multiplets.
\begin{itemize}
\item
Suppose \(K\) has minimal positive \(3\)-class rank \(\varrho_3=1\),
then a heterogeneous multiplet with \textit{free} prime conductor \(q\),
having \(p\)-defect \(\delta_p(q)=0\)
\cite{Ma2014},
possesses
\begin{equation}
\label{eqn:Rank1}
m=m_3(d_K,1)+m_3(d_K,q)=\frac{1}{2}(3^{\varrho_3+1}-1)=\frac{1}{2}(3^{1+1}-1)=\frac{1}{2}(9-1)=4
\end{equation}
members,
being a \textit{quartet}, and
consists of an unramified homogeneous \textit{singulet} \(N_{1,1}\),
since \(m_3(d_K,1)=\frac{1}{2}(3^{\varrho_3}-1)=\frac{1}{2}(3^{1}-1)=1\),
and a ramified homogeneous \textit{triplet} \((N_{q,1},N_{q,2},N_{q,3})\),
since \(m_3(d_K,q)=4-1=3\).
The situation is illustrated in Figure
\ref{fig:Quartet}.
\item
Assume that \(K\) has \(3\)-class rank \(\varrho_3=2\)
and \(q\) is free with \(p\)-defect \(\delta_p(q)=0\),
then a heterogeneous multiplet with conductor \(q\)
possesses
\begin{equation}
\label{eqn:Rank2}
m=m_3(d_K,1)+m_3(d_K,q)=\frac{1}{2}(3^{\varrho_3+1}-1)=\frac{1}{2}(3^{2+1}-1)=\frac{1}{2}(27-1)=13
\end{equation}
members,
being a \textit{tridecuplet}, and
consists of an unramified homogeneous \textit{quartet} \((N_{1,1},\ldots,\) \(N_{1,4})\),
since \(m_3(d_K,1)=\frac{1}{2}(3^{\varrho_3}-1)=\frac{1}{2}(3^{2}-1)=4\),
and a ramified homogeneous \textit{nonet} \((N_{q,1},\ldots,N_{q,9})\),
since \(m_3(d_K,q)=13-4=9\).
The situation is illustrated in Figure
\ref{fig:Tridecuplet}.
\end{itemize}
\end{example}


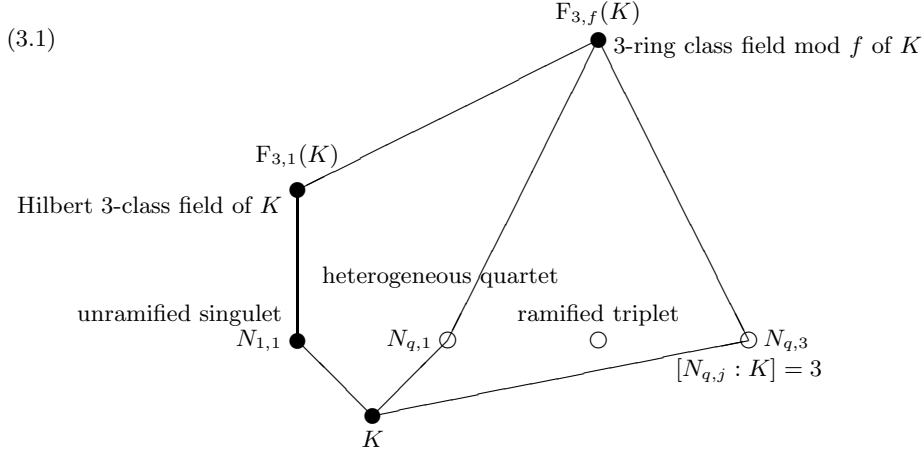
\begin{figure}[ht]
\caption{Heterogeneous Quartet over \(K\)}
\label{fig:Quartet}

{\small

\setlength{\unitlength}{1.0cm}
\begin{picture}(10,6)(-8,-9)



\put(-5,-9){\circle*{0.2}}
\put(-5,-9.2){\makebox(0,0)[ct]{\(K\)}}

\put(-5,-9){\line(-1,1){1}}
\put(-5,-9){\line(1,1){1}}
\put(-5,-9){\line(5,1){5}}

\put(-4.1,-7.3){\makebox(0,0)[cb]{heterogeneous quartet}}
\put(-6.2,-7.8){\makebox(0,0)[rb]{unramified singulet}}
\put(-6,-8){\circle*{0.2}}
\put(-6.2,-8){\makebox(0,0)[rc]{\(N_{1,1}\)}}

\put(-6,-8){\line(0,1){2}}

\put(-9.2,-4){\makebox(0,0)[rc]{(3.1)}}
\put(-6,-6){\circle*{0.2}}
\put(-6,-5.7){\makebox(0,0)[cb]{\(\mathrm{F}_{3,1}(K)\)}}
\put(-6.2,-6.2){\makebox(0,0)[rc]{Hilbert \(3\)-class field of \(K\)}}



\put(0,-8.4){\makebox(0,0)[cc]{\(\lbrack N_{q,j}:K\rbrack=3\)}}
\put(-2,-4){\line(-2,-1){4}}
\put(-2,-4){\line(-1,-2){2}}
\put(-2,-4){\line(1,-2){2}}

\put(-2,-7.8){\makebox(0,0)[cb]{ramified triplet}}
\put(-4,-8){\circle{0.2}}
\put(-2,-8){\circle{0.2}}
\put(0,-8){\circle{0.2}}
\put(-4.2,-8){\makebox(0,0)[rc]{\(N_{q,1}\)}}
\put(0.2,-8){\makebox(0,0)[lc]{\(N_{q,3}\)}}

\put(-2,-4){\circle*{0.2}}
\put(-2,-3.8){\makebox(0,0)[cb]{\(\mathrm{F}_{3,f}(K)\)}}
\put(-1.8,-4.1){\makebox(0,0)[lc]{\(3\)-ring class field mod \(f\) of \(K\)}}


\end{picture}

}

\end{figure}


\begin{figure}[ht]
\caption{Heterogeneous Tridecuplet over \(K\)}
\label{fig:Tridecuplet}

{\small

\setlength{\unitlength}{1.0cm}
\begin{picture}(10,7)(-8,-9)



\put(-5,-9){\circle*{0.2}}
\put(-5,-9.2){\makebox(0,0)[ct]{\(K\)}}

\put(-5,-9){\line(-5,1){5}}
\put(-5,-9){\line(-1,1){1}}
\put(-5,-9){\line(1,1){1}}
\put(-5,-9){\line(5,1){5}}

\put(-10,-8){\line(2,1){2}}
\put(-6,-8){\line(-2,1){2}}

\put(-5.5,-7.3){\makebox(0,0)[cb]{heterogeneous tridecuplet}}
\put(-8,-7.9){\makebox(0,0)[cb]{unramified quartet}}
\put(-10,-8){\circle*{0.2}}
\put(-8.7,-8){\circle*{0.2}}
\put(-7.3,-8){\circle*{0.2}}
\put(-6,-8){\circle*{0.2}}
\put(-10.2,-8){\makebox(0,0)[rc]{\(N_{1,1}\)}}
\put(-5.8,-8){\makebox(0,0)[lc]{\(N_{1,4}\)}}

\put(-8,-7){\circle*{0.2}}

\put(-8,-7){\line(0,1){2}}

\put(-9.2,-3){\makebox(0,0)[rc]{(3.2)}}
\put(-8,-5){\circle*{0.2}}
\put(-8,-4.8){\makebox(0,0)[cb]{\(\mathrm{F}_{3,1}(K)\)}}
\put(-7.8,-5.2){\makebox(0,0)[lc]{Hilbert \(3\)-class field of \(K\)}}



\put(0,-8.4){\makebox(0,0)[cc]{\(\lbrack N_{q,j}:K\rbrack=3\)}}
\put(0,-3){\line(-4,-1){8}}
\put(0,-3){\line(-4,-5){4}}
\put(0,-3){\line(0,-1){5}}

\put(-2,-7.8){\makebox(0,0)[cb]{ramified nonet}}
\put(-4,-8){\circle{0.2}}
\put(-3.5,-8){\circle{0.2}}
\put(-3,-8){\circle{0.2}}
\put(-2.5,-8){\circle{0.2}}
\put(-2,-8){\circle{0.2}}
\put(-1.5,-8){\circle{0.2}}
\put(-1,-8){\circle{0.2}}
\put(-0.5,-8){\circle{0.2}}
\put(0,-8){\circle{0.2}}
\put(-4.2,-8){\makebox(0,0)[rc]{\(N_{q,1}\)}}
\put(0.2,-8){\makebox(0,0)[lc]{\(N_{q,9}\)}}

\put(0,-3){\circle*{0.2}}
\put(0,-2.8){\makebox(0,0)[cb]{\(\mathrm{F}_{3,f}(K)\)}}
\put(0.2,-3.1){\makebox(0,0)[lc]{\(3\)-ring class field mod \(f\) of \(K\)}}


\end{picture}

}

\end{figure}
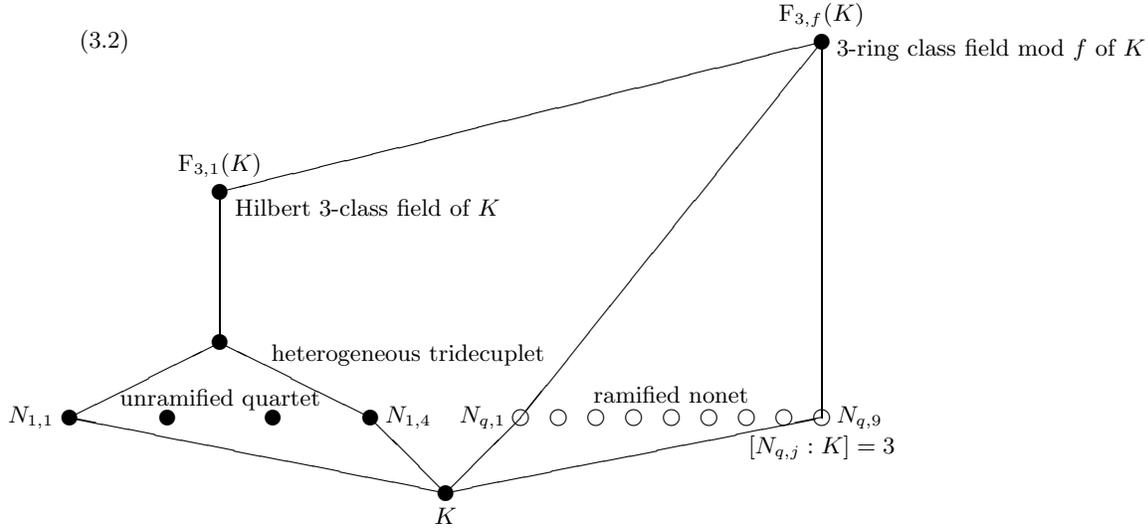


\begin{definition}
\label{dfn:Pattern}
The \textit{generalized} or \textit{extended Artin pattern} \(\mathrm{AP}(\mathbf{M})\)
of the heterogeneous multiplet \(\mathbf{M}=((N_{1,i})_i,(N_{q,j})_j)\)
consists of the kernels and targets of the \(p\)-class extension homomorphisms
\(T_{K,N_{1,i}}:\,\mathrm{Cl}_p(K)\to\mathrm{Cl}_p(N_{1,i})\), \(1\le i\le m(1)\), and
\(T_{K,N_{q,j}}:\,\mathrm{Cl}_p(K)\to\mathrm{Cl}_p(N_{q,j})\), \(1\le j\le m(q)\),
where we abbreviate \(m(f)=m_p(d_K,f)\), since \(p\) and \(K\) are fixed. Explicitly:
\begin{equation}
\label{eqn:Pattern}
\begin{aligned}
\mathrm{AP}(\mathbf{M}) &= ((\varkappa_u(\mathbf{M}),\tau_u(\mathbf{M})),(\varkappa_r(\mathbf{M}),\tau_r(\mathbf{M}))), \text{ where} \\
\varkappa_u(\mathbf{M}) &= (\ker(T_{K,N_{1,i}}))_i,\ \tau_u(\mathbf{M})=(\mathrm{Cl}_p(N_{1,i}))_i\text{ (\(u\ldots\)unramified), and} \\
\varkappa_r(\mathbf{M}) &= (\ker(T_{K,N_{q,j}}))_j,\ \tau_r(\mathbf{M})=(\mathrm{Cl}_p(N_{q,j}))_j\text{ (\(r\ldots\)ramified)}.
\end{aligned}
\end{equation}
\end{definition}


\begin{example}
\label{exm:Rank1}
\noindent
Let \(p=3\) be the smallest odd prime,
\(K\) be a \textit{real} quadratic field with \(3\)-class rank \(\varrho_3=1\),
and \(f=2\) be a free admissible prime conductor for \(K\).
According to Formula
\eqref{eqn:Rank1}
in Example
\ref{exm:Multiplet},
\(K\) gives rise to
a singulet \(N_{1,1}\) of unramified absolute \(3\)-class fields, and
a triplet \((N_{2,1},N_{2,2},N_{2,3})\) of ramified \(3\)-ring class fields modulo \(2\).
The generalized Artin pattern of the heterogeneous multiplet
\((N_{1,1};(N_{2,j})_{j=1}^3)\) is
\((\ker(T_{K,N_{1,1}}),\mathrm{Cl}_3(N_{1,1});(\ker(T_{K,N_{2,j}}),\mathrm{Cl}_3(N_{2,j}))_{j=1}^3)\).
\end{example}


\section{Multiplets of objects and associated multiplets of invariants}
\label{s:ObjectsInvariants}
\noindent
Since we want to supplement our description of dihedral fields
with differential principal factorization (DPF) types,
additionally to the generalized Artin patterns,
we briefly explain that we shall have to deal with
two distinct kinds of closely related multiplets.
Suppose, we have a mapping,
\begin{equation}
\label{eqn:ObjectsInvariants}
I:\,\mathrm{Objects}\to\mathrm{Invariants},\ X\mapsto I(X),
\end{equation}
from certain objects to invariants of these objects,
then a \textit{multiplet of objects}, \((X_1,\ldots,X_m)\),
with some non-negative integer \(m\),
will be mapped to an \textit{associated multiplet of invariants},
\(I(X_1,\ldots,X_m)=(I(X_1),\ldots,I(X_m))\).

\begin{example}
\label{exm:ObjectsInvariants}
\noindent
In Theorem
\ref{thm:MainCubic}
we shall see that the DPF type of pure cubic fields \(L=\mathbb{Q}(\sqrt[3]{d})\)
is a mapping \(L\mapsto T(L)\in\lbrace\alpha,\beta,\gamma\rbrace\).
Therefore, a multiplet \((L_1,\ldots,L_m)\) of pure cubic fields
has an associated multiplet of DPF types \((T(L_1),\ldots,T(L_m))\hat{=}(\alpha^x,\beta^y,\gamma^z)\) with \(x+y+z=m\).
\end{example}


\section{Trichotomy of primitive ambiguous principal ideals}
\label{s:Trichotomy}
\noindent
Our intention in this section is
to supplement the generalized Artin pattern by \textit{differential principal factorization} (DPF) \textit{types}
and to establish
a common theoretical framework for
the classification
\begin{itemize}
\item
of dihedral fields \(N/\mathbb{Q}\) of degree \(2p\) with an odd prime \(p\),
viewed as \(p\)-ring class fields over a quadratic field \(K\),
and
\item
of pure metacyclic fields \(N=K(\sqrt[p]{D})\) of degree \((p-1)\cdot p\) with an odd prime \(p\),
viewed as Kummer extensions of a cyclotomic field \(K=\mathbb{Q}(\zeta_p)\),
\end{itemize}
by the following arithmetical invariants:
\begin{enumerate}
\item
the \(\mathbb{F}_p\)-dimensions of subspaces
of the space \(\mathcal{P}_{N/K}/\mathcal{P}_K\) of primitive ambiguous principal ideals,
which are also called \textit{differential principal factors}, of \(N/K\),
\item
the \textit{capitulation kernel} \(\ker(T_{N/K})\)
of \(T_{N/K}:\,\mathrm{Cl}_p(K)\to\mathrm{Cl}_p(N)\),
the transfer homomorphism of \(p\)-classes from \(K\) to \(N\), and
\item
the \textit{Galois cohomology} \(\mathrm{H}^0(G,U_N)\), \(\mathrm{H}^1(G,U_N)\) of the unit group \(U_N\)
as a module over the automorphism group \(G=\mathrm{Gal}(N/K)\simeq C_p\).
\end{enumerate}

\begin{figure}[ht]
\caption{Dihedral and Metacyclic Situation}
\label{fig:DihedralMetacyclic}

{\small

\setlength{\unitlength}{1.0cm}
\begin{picture}(12,5)(-7,-9)



\put(-6,-9){\circle*{0.2}}
\put(-6,-9.2){\makebox(0,0)[ct]{\(\mathbb{Q}\)}}
\put(-7,-9){\makebox(0,0)[rc]{rational field}}

\put(-6,-9){\line(2,1){2}}
\put(-5,-8.7){\makebox(0,0)[lt]{\(\lbrack K:\mathbb{Q}\rbrack=2\)}}

\put(-4,-8){\circle*{0.2}}
\put(-4,-8.2){\makebox(0,0)[lt]{\(K=\mathbb{Q}(\sqrt{d})\)}}
\put(-3,-8){\makebox(0,0)[lc]{quadratic field}}


\put(-6.2,-7.5){\makebox(0,0)[rc]{\(\lbrack L:\mathbb{Q}\rbrack=p\)}}
\put(-6,-9){\line(0,1){3}}
\put(-4,-8){\line(0,1){3}}
\put(-3.8,-6.5){\makebox(0,0)[lc]{cyclic extension}}



\put(-6,-6){\circle{0.2}}
\put(-6,-5.8){\makebox(0,0)[cb]{\(L\)}}
\put(-6.1,-6.2){\makebox(0,0)[rt]{\(L_{1},\ldots,L_{p-1}\)}}
\put(-7,-6){\makebox(0,0)[rc]{\(p\) conjugates}}

\put(-6,-6){\line(2,1){2}}

\put(-4,-5){\circle*{0.2}}
\put(-4,-4.8){\makebox(0,0)[cb]{\(N=L\cdot K\)}}
\put(-3,-5){\makebox(0,0)[lc]{dihedral field}}
\put(-3,-5.5){\makebox(0,0)[lc]{of degree \(2p\)}}

\put(-0.9,-9.4){\line(0,1){5}}


\put(2,-9){\circle*{0.2}}
\put(2,-9.2){\makebox(0,0)[ct]{\(\mathbb{Q}\)}}
\put(1,-9){\makebox(0,0)[rc]{rational field}}

\put(2,-9){\line(2,1){2}}
\put(3,-8.7){\makebox(0,0)[lt]{\(\lbrack K:\mathbb{Q}\rbrack=p-1\)}}
\put(3,-8.5){\circle*{0.2}}

\put(4,-8){\circle*{0.2}}
\put(4,-8.2){\makebox(0,0)[lt]{\(K=\mathbb{Q}(\zeta_p)\)}}
\put(5,-8){\makebox(0,0)[lc]{cyclotomic field}}


\put(1.8,-7.5){\makebox(0,0)[rc]{\(\lbrack L:\mathbb{Q}\rbrack=p\)}}
\put(2,-9){\line(0,1){3}}
\put(3,-6.5){\vector(0,1){0.8}}
\put(3,-6.8){\makebox(0,0)[cc]{intermediate}}
\put(3,-7.2){\makebox(0,0)[cc]{fields}}
\put(3,-7.5){\vector(0,-1){0.8}}
\put(4,-8){\line(0,1){3}}
\put(4.2,-6.5){\makebox(0,0)[lc]{Kummer extension}}



\put(2,-6){\circle{0.2}}
\put(1.8,-5.7){\makebox(0,0)[cb]{\(L=\mathbb{Q}(\sqrt[p]{D})\)}}
\put(1.9,-6.2){\makebox(0,0)[rt]{\(L_{1},\ldots,L_{p-1}\)}}
\put(1,-6){\makebox(0,0)[rc]{\(p\) conjugates}}

\put(2,-6){\line(2,1){2}}
\put(3,-5.5){\circle{0.2}}

\put(4,-5){\circle*{0.2}}
\put(4,-4.8){\makebox(0,0)[cb]{\(N=L\cdot K\)}}
\put(5,-5){\makebox(0,0)[lc]{metacyclic field}}
\put(5,-5.5){\makebox(0,0)[lc]{of degree \((p-1)p\)}}


\end{picture}

}

\end{figure}
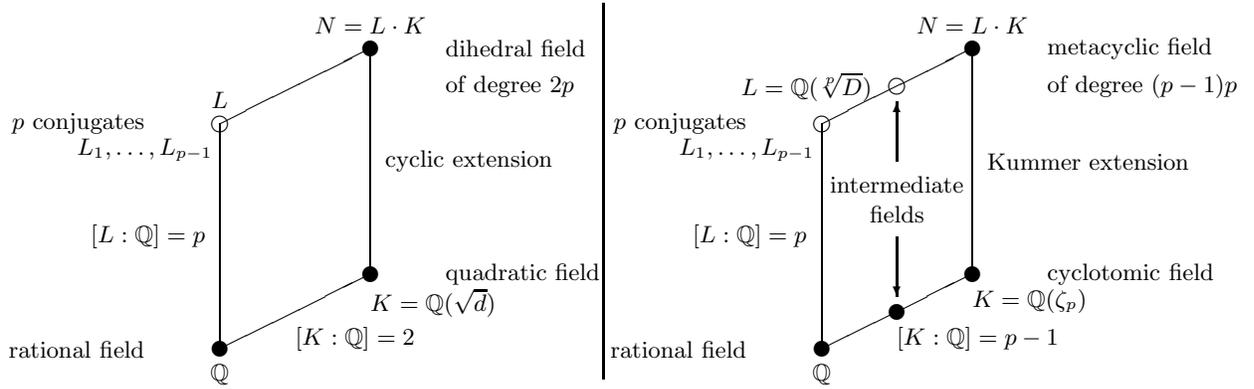


\subsection{Primitive ambiguous ideals}
\label{ss:PrmAmbIdl}
\noindent
Let \(p\ge 2\) be a prime number,
and \(N/K\) be a relative extension of number fields with degree \(p\)
(\textit{not} necessarily Galois).

\begin{definition}
\label{dfn:Ambiguous}
The group \(\mathcal{I}_N\) of fractional ideals of \(N\)
contains the \textit{subgroup of ambiguous ideals} of \(N/K\),
denoted by the symbol
\(\mathcal{I}_{N/K}:=\lbrace\mathfrak{A}\in\mathcal{I}_N\mid\mathfrak{A}^p\in\mathcal{I}_K\rbrace\).
The quotient \(\mathcal{I}_{N/K}/\mathcal{I}_K\) is called
the \(\mathbb{F}_p\)-\textit{vector space of primitive ambiguous ideals} of \(N/K\).
\end{definition}

\begin{proposition}
\label{prp:PrimitiveAmbiguous}
Let \(\mathfrak{L}_1,\ldots,\mathfrak{L}_t\) be the totally ramified prime ideals of \(N/K\),
then a basis and the dimension of the space \(\mathcal{I}_{N/K}/\mathcal{I}_K\) over \(\mathbb{F}_p\) are finite and given by
\begin{equation}
\label{eqn:PrimitiveAmbiguous}
\mathcal{I}_{N/K}/\mathcal{I}_K
\simeq\prod_{i=1}^t\,(\langle\mathfrak{L}_i\rangle/\langle\mathfrak{L}_i^p\rangle)
\simeq\mathbb{F}_p^t, \quad
\dim_{\mathbb{F}_p}(\mathcal{I}_{N/K}/\mathcal{I}_K)=t,
\end{equation}
whereas \(\mathcal{I}_{N/K}\) is an \textit{infinite} abelian group containing \(\mathcal{I}_K\).
\end{proposition}

\begin{proof}
According to the definition of \(\mathcal{I}_{N/K}\),
the quotient \(\mathcal{I}_{N/K}/\mathcal{I}_K\) is an \textit{elementary} abelian \(p\)-group.
By the decomposition law for prime ideals of \(K\) in \(N\), the space
\(\mathcal{I}_{N/K}/\mathcal{I}_K\)
is generated by the \textit{totally ramified} prime ideals (with ramification index \(e=p\)) of \(N/K\),
that is to say
\(\mathcal{I}_{N/K}=\langle\mathfrak{L}\in\mathbb{P}_N\mid\mathfrak{L}^p\in\mathbb{P}_K\rangle\mathcal{I}_K\).
According to the theorem on prime ideals dividing the discriminant,
the number \(t\) of totally ramified prime ideals \(\mathfrak{L}_1,\ldots,\mathfrak{L}_t\) of \(N/K\)
is \textit{finite}.
\end{proof}


If \(L\) is another subfield of \(N\)
such that \(N=L\cdot K\) is the compositum of \(L\) and \(K\),
and \(N/L\) is of degree \(q\) \textit{coprime} to \(p\),
then the relative norm homomorphism \(N_{N/L}\) induces an \textit{epimorphism}
\begin{equation}
\label{eqn:InducedNorm}
N_{N/L}:\,\mathcal{I}_{N/K}/\mathcal{I}_K\to\mathcal{I}_{L/F}/\mathcal{I}_F,
\end{equation}
where \(F:=L\cap K\) denotes the intersection of \(L\) and \(K\) in Figure
\ref{fig:RelativeSubfields}.
Thus, by the isomorphism theorem,
we have proved:
\begin{theorem}
\label{thm:QualitativeDichotomy}
There are the following two isomorphisms between \(\mathbb{F}_p\)-vector spaces:
\begin{equation}
\label{eqn:Dichotomy}
\begin{aligned}
(\mathcal{I}_{N/K}/\mathcal{I}_K)/\ker(N_{N/L}) &\simeq \mathcal{I}_{L/F}/\mathcal{I}_F \quad \text{(quotient)}, \\
\mathcal{I}_{N/K}/\mathcal{I}_K &\simeq (\mathcal{I}_{L/F}/\mathcal{I}_F)\times\ker(N_{N/L}) \quad \text{(direct product)}.
\end{aligned}
\end{equation}
\end{theorem}


\begin{figure}[ht]
\caption{Hasse Subfield Diagram of \(N/F\)}
\label{fig:RelativeSubfields}

{\small

\setlength{\unitlength}{1.0cm}
\begin{picture}(5,5)(-7,-9)



\put(-6,-9){\circle*{0.2}}
\put(-6,-9.2){\makebox(0,0)[ct]{\(F=L\cap K\)}}
\put(-7,-9){\makebox(0,0)[rc]{base field}}

\put(-6,-9){\line(2,1){2}}
\put(-5,-8.7){\makebox(0,0)[lt]{\(\lbrack K:F\rbrack=q\)}}

\put(-4,-8){\circle*{0.2}}
\put(-4,-8.2){\makebox(0,0)[ct]{\(K\)}}
\put(-3,-8){\makebox(0,0)[lc]{field of degree \(q\)}}


\put(-6.2,-7.5){\makebox(0,0)[rc]{\(\lbrack L:F\rbrack=p\)}}
\put(-6,-9){\line(0,1){3}}
\put(-4,-8){\line(0,1){3}}



\put(-6,-6){\circle*{0.2}}
\put(-6,-5.8){\makebox(0,0)[cb]{\(L\)}}
\put(-7,-6){\makebox(0,0)[rc]{field of degree \(p\)}}

\put(-6,-6){\line(2,1){2}}

\put(-4,-5){\circle*{0.2}}
\put(-4,-4.8){\makebox(0,0)[cb]{\(N=L\cdot K\)}}
\put(-3,-5){\makebox(0,0)[lc]{compositum of degree \(p\cdot q\)}}


\end{picture}

}

\end{figure}
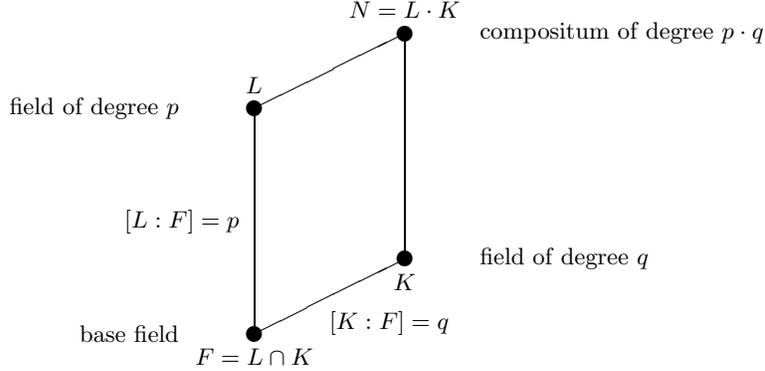


\begin{definition}
\label{dfn:Dichotomy}
Since the relative different of \(N/K\) is essentially given by
\(\mathfrak{D}_{N/K}=\prod_{i=1}^t\,\mathfrak{L}_i^{p-1}\)
the space \(\mathcal{I}_{N/K}/\mathcal{I}_K\simeq\prod_{i=1}^t\,(\langle\mathfrak{L}_i\rangle/\langle\mathfrak{L}_i^p\rangle)\)
of primitive ambiguous ideals of \(N/K\) is also called
the space of \textit{differential factors} of \(N/K\).
The two subspaces in the direct product decomposition of
\(\mathcal{I}_{N/K}/\mathcal{I}_K\) in formula
\eqref{eqn:Dichotomy}
are called \\
subspace \(\mathcal{I}_{L/F}/\mathcal{I}_F\) of \textit{absolute} differential factors of \(L/F\) and \\
subspace \(\ker(N_{N/L})\) of \textit{relative} differential factors of \(N/K\).
\end{definition}


\subsection{Splitting off the norm kernel}
\label{ss:NormKernel}
\noindent
The second isomorphism in formula
\eqref{eqn:Dichotomy}
causes a \textit{dichotomic decomposition}
of the space \(\mathcal{I}_{N/K}/\mathcal{I}_K\) of primitive ambiguous ideals of \(N/K\)
into two components, whose dimensions can be given
under the following conditions:

\begin{theorem}
\label{thm:QuantitativeDichotomy}
Let \(p\ge 3\) be an odd prime and put \(q=2\).
Among the prime ideals of \(L\) which are totally ramified over \(F\),
denote by \(\mathfrak{p}_1,\ldots,\mathfrak{p}_s\) those which split in \(N\),
\(\mathfrak{p}_i\mathcal{O}_N=\mathfrak{P}_i\mathfrak{P}_i^\prime\) for \(1\le i\le s\),
and by \(\mathfrak{q}_1,\ldots,\mathfrak{q}_n\) those which remain inert in \(N\),
\(\mathfrak{q}_j\mathcal{O}_N=\mathfrak{Q}_j\) for \(1\le j\le n\).
Then the space \(\mathcal{I}_{N/K}/\mathcal{I}_K\) of primitive ambiguous ideals of \(N/K\)
is the direct product of
the subspace \(\mathcal{I}_{L/F}/\mathcal{I}_F\) of \textbf{absolute differential factors} of \(L/F\)
and the subspace \(\ker(N_{N/L})\) of \textbf{relative differential factors} of \(N/K\),
whose dimensions over \(\mathbb{F}_p\) are given by
\begin{equation}
\label{eqn:AbsAndRel}
\begin{aligned}
\mathcal{I}_{L/F}/\mathcal{I}_F
&\simeq \prod_{i=1}^{s}\,(\langle\mathfrak{p}_i\rangle/\langle\mathfrak{p}_i^p\rangle)
\times\prod_{j=1}^{n}\,(\langle\mathfrak{q}_j\rangle/\langle\mathfrak{q}_j^p\rangle)
\simeq\mathbb{F}_p^{s+n}, \\
\ker(N_{N/L})
&\simeq \prod_{i=1}^{s}\,\Bigl(\langle\mathfrak{P}_i(\mathfrak{P}_i^\prime)^{p-1}\rangle/\langle(\mathfrak{P}_i(\mathfrak{P}_i^\prime)^{p-1})^p\rangle\Bigr)
\simeq\mathbb{F}_p^s.
\end{aligned}
\end{equation}
Consequently, the complete space of differential factors has dimension \(\dim_{\mathbb{F}_p}(\mathcal{I}_{N/K}/\mathcal{I}_K)=n+2s\).
\end{theorem}

\begin{proof}
Whereas the qualitative formula
\eqref{eqn:Dichotomy}
is valid for any prime \(p\ge 2\) and any integer \(q>1\) with \(\gcd(p,q)=1\),
the quantitative description of the norm kernel \(\ker(N_{N/L})\) is only feasible
if we put \(q=2\) and therefore have to select an odd prime \(p\ge 3\).
Replacing \(N\) by \(L\) and \(K\) by \(F\) in formula
\eqref{eqn:Dichotomy},
we get \(t=n+s\) and thus the first isomorphism of formula
\eqref{eqn:AbsAndRel}.
For \(N\) and \(K\), however, we obtain \(t=n+2s\).
We point out that, if \(s=0\),
that is, if none of the totally ramified primes of \(L/F\) splits in \(N\),
then the induced norm mapping \(N_{N/L}\) in formula
\eqref{eqn:InducedNorm}
is an isomorphism.
For the constitution of the norm kernel, see
\cite[Thm. 3.4 and Cor. 3.3(3)]{Ma2018}.
\end{proof}


\subsection{Primitive ambiguous principal ideals}
\label{ss:PrmAmbPrcIdl}
\noindent
The preceding result concerned \textit{primitive ambiguous} \textbf{ideals} of \(N/K\),
which can be interpreted as ideal factors of the \textit{relative different} \(\mathfrak{D}_{N/K}\).
Formula
\eqref{eqn:PrimitiveAmbiguous}
and Theorem
\ref{thm:QuantitativeDichotomy}
show that the \(\mathbb{F}_p\)-dimension of the space \(\mathcal{I}_{N/K}/\mathcal{I}_K\)
increases indefinitely with the number \(t\) of totally ramified primes of \(N/K\).

\noindent
Now we restrict our attention to the space \(\mathcal{P}_{N/K}/\mathcal{P}_K\)
of \textit{primitive ambiguous} \textbf{principal ideals} or \textit{differential principal factors} (DPF) of \(N/K\).
We shall see that fundamental constraints from Galois cohomology
prohibit an infinite growth of its dimension over \(\mathbb{F}_p\),
for quadratic fields \(K\).


\subsection{Splitting off the capitulation kernel}
\label{ss:CapitulationKernel}
\noindent
We have to cope with a difficulty
which arises in the case of a non-trivial class group
\(\mathrm{Cl}(K)=\mathcal{I}_K/\mathcal{P}_K>1\),
because then \(\mathcal{P}_{N/K}/\mathcal{P}_K\) cannot be viewed as a subgroup of \(\mathcal{I}_{N/K}/\mathcal{I}_K\).
Therefore we must separate the \textit{capitulation kernel} of \(N/K\),
that is the kernel of the \textit{transfer} homomorphism
\(T_{N/K}:\,\mathrm{Cl}(K)\to\mathrm{Cl}(N)\), \(\mathfrak{a}\cdot\mathcal{P}_K\mapsto(\mathfrak{a}\mathcal{O}_N)\cdot\mathcal{P}_N\),
which extends classes of \(K\) to classes of \(N\):
\begin{equation}
\label{eqn:Capitulation}
\ker(T_{N/K})
=\lbrace\mathfrak{a}\cdot\mathcal{P}_K\mid(\exists\,A\in N)\,\mathfrak{a}\mathcal{O}_N=A\mathcal{O}_N\rbrace
=(\mathcal{I}_K\cap\mathcal{P}_N)/\mathcal{P}_K.
\end{equation}
On the one hand,
\(\ker(T_{N/K})=(\mathcal{I}_K\cap\mathcal{P}_N)/\mathcal{P}_K\) is a subgroup of \(\mathcal{I}_K/\mathcal{P}_K=\mathrm{Cl}(K)\),
consisting of capitulating ideal classes of \(K\).
On the other hand,
since \(\mathcal{I}_K\le\mathcal{I}_{N/K}\) consists of ambiguous ideals of \(N/K\),
\(\ker(T_{N/K})=(\mathcal{I}_K\cap\mathcal{P}_N)/\mathcal{P}_K\) is a subgroup of \(\mathcal{P}_{N/K}/\mathcal{P}_K\),
consisting of special primitive ambiguous principal ideals of \(N/K\),
and we can form the quotient
\begin{equation}
\label{eqn:QuotientSeparation}
(\mathcal{P}_{N/K}/\mathcal{P}_K)/\bigl((\mathcal{I}_K\cap\mathcal{P}_N)/\mathcal{P}_K\bigr)
\simeq\mathcal{P}_{N/K}/(\mathcal{I}_K\cap\mathcal{P}_N)=\mathcal{P}_{N/K}/(\mathcal{I}_K\cap\mathcal{P}_{N/K})
\simeq(\mathcal{P}_{N/K}\cdot\mathcal{I}_K)/\mathcal{I}_K.
\end{equation}
This quotient relation of \(\mathbb{F}_p\)-vector spaces is equivalent with a direct product relation
\begin{equation}
\label{eqn:ProductSeparation}
\mathcal{P}_{N/K}/\mathcal{P}_K
\simeq(\mathcal{P}_{N/K}\cdot\mathcal{I}_K)/\mathcal{I}_K\times\ker(T_{N/K}).
\end{equation}
Since \((\mathcal{P}_{N/K}\cdot\mathcal{I}_K)/\mathcal{I}_K\le\mathcal{I}_{N/K}/\mathcal{I}_K\) is an actual inclusion,
the factorization of \(\mathcal{I}_{N/K}/\mathcal{I}_K\) in formula
\eqref{eqn:Dichotomy}
restricts to a factorization
\begin{equation}
\label{eqn:PrincipalDichotomy}
(\mathcal{P}_{N/K}\cdot\mathcal{I}_K)/\mathcal{I}_K
\simeq(\mathcal{P}_{L/F}/\mathcal{P}_F)\times\Bigl(\ker(N_{N/L})\cap\bigl((\mathcal{P}_{N/K}\cdot\mathcal{I}_K)/\mathcal{I}_K\bigr)\Bigr),
\end{equation}
provided that \(F\) is a field with trivial class group \(\mathrm{Cl}(F)\),
that is \(\mathcal{I}_F=\mathcal{P}_F\)
and thus \(\mathcal{P}_{L/F}/\mathcal{P}_F\le\mathcal{I}_{L/F}/\mathcal{I}_F\).
Combining the formulas
\eqref{eqn:ProductSeparation}
and
\eqref{eqn:PrincipalDichotomy}
for the rational base field \(F=\mathbb{Q}\) ,
we obtain:


\begin{theorem}
\label{thm:Trichotomy}
There is a \textbf{trichotomic decomposition}
of the space \(\mathcal{P}_{N/K}/\mathcal{P}_K\) of differential principal factors of \(N/K\)
into three components,
\begin{equation}
\label{eqn:Trichotomy}
\mathcal{P}_{N/K}/\mathcal{P}_K\simeq
\mathcal{P}_{L/\mathbb{Q}}/\mathcal{P}_{\mathbb{Q}}
\times\Bigl(\ker(N_{N/L})\cap\bigl((\mathcal{P}_{N/K}\mathcal{I}_K)/\mathcal{I}_K\bigr)\Bigr)
\times\ker(T_{N/K}),
\end{equation}
the \textbf{absolute principal factors}, \(\mathcal{P}_{L/\mathbb{Q}}/\mathcal{P}_{\mathbb{Q}}\), of \(L/\mathbb{Q}\), \\
the \textbf{relative principal factors}, \(\ker(N_{N/L})\cap\bigl((\mathcal{P}_{N/K}\mathcal{I}_K)/\mathcal{I}_K\bigr)\), of \(N/K\), and \\
the \textbf{capitulation kernel}, \(\ker(T_{N/K})\), of \(N/K\).
\end{theorem}


\subsection{Galois cohomology}
\label{ss:GaloisCohomology}
\noindent
For establishing a quantitative version of the qualitative formula
\eqref{eqn:Trichotomy},
we suppose that \(N/K\) is a cyclic relative extension of odd prime degree \(p\)
and we use the Galois cohomology of the unit group \(U_N\)
as a module over the automorphism group \(G=\mathrm{Gal}(N/K)=\langle\sigma\rangle\simeq C_p\).
In fact, we combine a theorem of Iwasawa
\cite{Iw}
on the first cohomology \(\mathrm{H}^1(G,U_N)\)
with a theorem of Hasse
\cite{Ha1927}
on the Herbrand quotient of \(U_N\)
\cite{Hb},
and we use Dirichlet's theorem on the torsion-free unit rank of \(K\):
\begin{equation}
\label{eqn:Herbrand}
\begin{aligned}
\mathrm{H}^1(G,U_N) &\simeq (U_N\cap\ker(N_{N/K}))/U_N^{\sigma-1}\simeq\mathcal{P}_{N/K}/\mathcal{P}_K\ \text{(Iwasawa)}, \\
\mathrm{H}^0(G,U_N) &\simeq U_K/N_{N/K}(U_N), \text{ with } (U_K:N_{N/K}(U_N))=p^U,\ 0\le U\le r_1+r_2-\theta, \\
\frac{\#\mathrm{H}^1(G,U_N)}{\#\mathrm{H}^0(G,U_N)} &= \lbrack N:K\rbrack=p \quad \text{(Hasse)},
\end{aligned}
\end{equation}
where \((r_1,r_2)\) is the signature of \(K\), and \(\theta=0\) if \(K\) contains the \(p\)th roots of unity, but \(\theta=1\) else.


\begin{corollary}
\label{cor:Trichotomy}
If \(N/K\) is cyclic of odd prime degree \(p\ge 3\),
then the \(\mathbb{F}_p\)-dimensions of the spaces of differential principal factors in Theorem
\ref{thm:Trichotomy}
are connected by the \textbf{fundamental equation}
\begin{equation}
\label{eqn:Dimensions}
U+1=A+R+C,\quad \text{where}
\end{equation}
\(A:=\dim_{\mathbb{F}_p}(\mathcal{P}_{L/\mathbb{Q}}/\mathcal{P}_{\mathbb{Q}})\) is the dimension of absolute principal factors, \\
\(R:=\dim_{\mathbb{F}_p}\Bigl(\ker(N_{N/L})\cap\bigl((\mathcal{P}_{N/K}\mathcal{I}_K)/\mathcal{I}_K\bigr)\Bigr)\)
is the dimension of relative principal factors, and \\
\(C:=\dim_{\mathbb{F}_p}(\ker(T_{N/K}))\) is the dimension of the capitulation karnel.
\end{corollary}


\begin{corollary}
\label{cor:Estimates}
Under the assumptions \(p\ge 3\), \(q=2\) of Theorem
\ref{thm:QualitativeDichotomy},
in particular for \(N\) dihedral of degree \(2p\),
the dimensions in Corollary
\ref{cor:Trichotomy}
are bounded by the following \textbf{estimates}
\begin{equation}
\label{eqn:Estimates}
0\le A\le\min(n+s,m), \quad
0\le R\le\min(s,m), \quad
0\le C\le\min(\varrho_p,m),
\end{equation}
where \(\varrho_p:=\mathrm{rank}_p(\mathrm{Cl}(K))\),
and \(m:=1+r_1+r_2-\theta\) denotes the cohomological maximum of \(U+1\).
In particular, \\
\(m=2\) for real quadratic \(K\) with \((r_1,r_2)=(2,0)\) or \(K=\mathbb{Q}(\sqrt{-3})\) if \(p=3\), \\
\(m=1\) for imaginary quadratic \(K\) with \((r_1,r_2)=(0,1)\), except \(K=\mathbb{Q}(\sqrt{-3})\) when \(p=3\).
\end{corollary}


\begin{remark}
\label{rmk:Estimates}
For \(N\) pure metacyclic of degree \((p-1)p\),
the space \(\mathcal{P}_{L/\mathbb{Q}}/\mathcal{P}_{\mathbb{Q}}\) of absolute principal factors
contains the one-dimensional subspace \(\Delta=\langle\sqrt[p]{D}\rangle\) generated by the \textit{radicals}, and thus
\begin{equation}
\label{eqn:EstimatesMet}
1\le A\le\min(t,m), \
0\le R\le m-1, \
0\le C\le\min(\varrho_p,m-1),
\end{equation}
where \(m=\frac{p+1}{2}\) for cyclotomic \(K\) with \((r_1,r_2)=(0,\frac{p-1}{2})\). \\
In particular \(C=0\) for a \textit{regular} prime \(p\), for instance \(p<37\).
\end{remark}


\begin{remark}
\label{rmk:Inclusion}
We mentioned that in general
\(\mathcal{P}_{N/K}/\mathcal{P}_K\) cannot be viewed as a subgroup of \(\mathcal{I}_{N/K}/\mathcal{I}_K\).
In fact, for a dihedral field \(N\) which is unramified with conductor \(f=1\) over \(K\),
we have \(n=s=0\),
consequently \(A=R=0\),
and \(\mathcal{I}_{N/K}/\mathcal{I}_K\simeq 0\) is the nullspace,
whereas \(\mathcal{P}_{N/K}/\mathcal{P}_K\simeq\ker(T_{N/K})\) is at least one-dimensional,
according to Hilbert's Theorem 94
\cite{Hi},
and at most two-dimensional by the estimate \(C\le\min(\varrho_p,m)\le\min(\varrho_p,2)\le 2\).
\end{remark}


In the next two sections, we apply the results of \S\S\
\ref{ss:PrmAmbIdl}
--
\ref{ss:GaloisCohomology}
to various extensions \(N/K\).


\subsection{Differential principal factorization (DPF) types of complex dihedral fields}
\label{ss:DihedralTypes}
\noindent
Let \(p\) be an odd prime.
We recall the classification theorem
for \textit{pure cubic} fields \(L=\mathbb{Q}(\sqrt[3]{D})\)
and their Galois closure \(N=\mathbb{Q}(\zeta_3,\sqrt[3]{D})\),
that is the metacyclic case \(p=3\).
The \textit{coarse} classification of \(N\)
according to the cohomological invariants \(U\) and \(A\) alone
is closely related to the
classification of \textit{simply real dihedral} fields of degree \(2p\) with any odd prime \(p\)
by Nicole Moser
\cite[Dfn. III.1 and Prop. III.3, p. 61]{Mo},
as illustrated in Figure
\ref{fig:MoserExtendedCubic}.
The coarse types \(\alpha\) and \(\beta\)
are completely analogous in both cases.
The additional type \(\gamma\) is required for pure cubic fields,
because there arises the possibility that the primitive cube root of unity \(\zeta_3\)
occurs as relative norm \(N_{N/K}(Z)\) of a unit \(Z\in U_N\).
Due to the existence of radicals in the pure cubic case,
the \(\mathbb{F}_p\)-dimension \(A\) of the vector space of absolute DPF
exceeds the corresponding dimension for simply real dihedral fields by one.


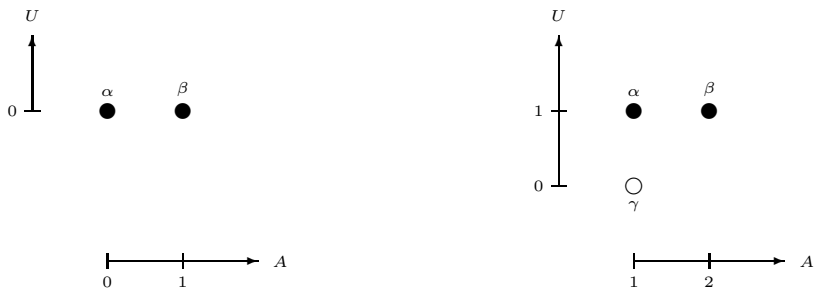
\begin{figure}[ht]
\caption{Classification of Simply Real Dihedral and Pure Cubic Fields}
\label{fig:MoserExtendedCubic}

{\tiny

\setlength{\unitlength}{1.0cm}
\begin{picture}(15,4)(-11,-9)




\put(-9,-5.8){\makebox(0,0)[cb]{\(U\)}}
\put(-9,-7){\vector(0,1){1}}
\put(-9,-7){\line(0,1){0}}
\multiput(-9.1,-7)(0,1){1}{\line(1,0){0.2}}

\put(-9.2,-7){\makebox(0,0)[rc]{\(0\)}}

\put(-5.8,-9){\makebox(0,0)[lc]{\(A\)}}
\put(-7,-9){\vector(1,0){1}}
\put(-8,-9){\line(1,0){1}}
\multiput(-8,-9.1)(1,0){2}{\line(0,1){0.2}}

\put(-8,-9.2){\makebox(0,0)[ct]{\(0\)}}
\put(-7,-9.2){\makebox(0,0)[ct]{\(1\)}}


\put(-8,-7){\circle*{0.2}}
\put(-8,-6.8){\makebox(0,0)[cb]{\(\alpha\)}}
\put(-7,-7){\circle*{0.2}}
\put(-7,-6.8){\makebox(0,0)[cb]{\(\beta\)}}




\put(-2,-5.8){\makebox(0,0)[cb]{\(U\)}}
\put(-2,-7){\vector(0,1){1}}
\put(-2,-8){\line(0,1){1}}
\multiput(-2.1,-8)(0,1){2}{\line(1,0){0.2}}

\put(-2.2,-7){\makebox(0,0)[rc]{\(1\)}}
\put(-2.2,-8){\makebox(0,0)[rc]{\(0\)}}

\put(1.2,-9){\makebox(0,0)[lc]{\(A\)}}
\put(0,-9){\vector(1,0){1}}
\put(-1,-9){\line(1,0){1}}
\multiput(-1,-9.1)(1,0){2}{\line(0,1){0.2}}

\put(-1,-9.2){\makebox(0,0)[ct]{\(1\)}}
\put(0,-9.2){\makebox(0,0)[ct]{\(2\)}}


\put(-1,-7){\circle*{0.2}}
\put(-1,-6.8){\makebox(0,0)[cb]{\(\alpha\)}}
\put(0,-7){\circle*{0.2}}
\put(0,-6.8){\makebox(0,0)[cb]{\(\beta\)}}

\put(-1,-8){\circle{0.2}}
\put(-1,-8.2){\makebox(0,0)[ct]{\(\gamma\)}}


\end{picture}
}
\end{figure}


\noindent
The \textit{fine} classification of \(N\)
according to the invariants \(U\), \(A\), \(R\) and \(C\)
in the simply real dihedral situation with \(U+1=A+R+C\)
splits type \(\alpha\) with \(A=0\) further in
type \(\alpha_1\) with \(C=1\) (capitulation) and
type \(\alpha_2\) with \(R=1\) (relative DPF).
In the pure cubic situation, however, no further splitting occurs,
since \(C=0\), and \(R=U+1-A\) is determined uniquely by \(U\) and \(A\) already.
We oppose the two classifications in the following theorems.

\begin{theorem}
\label{thm:MainComplex}
Each simply real dihedral field \(N/\mathbb{Q}\)
of absolute degree \(\lbrack N:\mathbb{Q}\rbrack=2p\) with an odd prime \(p\)
belongs to precisely one of the following \(3\) differential principal factorization types,
in dependence on the triplet \((A,R,C)\):


\renewcommand{\arraystretch}{1.1}

\begin{table}[ht]
\label{tbl:ComplexDPFTypes}
\begin{center}
\begin{tabular}{|c||c||c||ccc|}
\hline
 Type        & \(U\) & \(U+1=A+R+C\) & \(A\) & \(R\) & \(C\) \\
\hline
\(\alpha_1\) & \(0\) & \(1\) & \(0\) & \(0\) & \(1\) \\
\(\alpha_2\) & \(0\) & \(1\) & \(0\) & \(1\) & \(0\) \\
\(\beta\)    & \(0\) & \(1\) & \(1\) & \(0\) & \(0\) \\
\hline
\end{tabular}
\end{center}
\end{table}


\end{theorem}

\begin{proof}
Consequence of Cor.
\ref{cor:Trichotomy}
and 
\ref{cor:Estimates}.
See \cite[Dfn. III.1 and Prop. III.3, p. 61]{Mo} and \cite{Ma1991}.
\end{proof}

\begin{theorem}
\label{thm:MainCubic}
Each pure metacyclic field \(N=\mathbb{Q}(\zeta_3,\sqrt[3]{D})\)
of absolute degree \(\lbrack N:\mathbb{Q}\rbrack=6\)
with cube free radicand \(D\in\mathbb{Z}\), \(D\ge 2\),
belongs to precisely one of the following \(3\) differential principal factorization types,
in dependence on the invariant \(U\) and the pair \((A,R)\):


\renewcommand{\arraystretch}{1.1}

\begin{table}[ht]
\label{tbl:CubicDPFTypes}
\begin{center}
\begin{tabular}{|c||c||c||cc|}
\hline
 Type      & \(U\) & \(U+1=A+R\) & \(A\) & \(R\) \\
\hline
\(\alpha\) & \(1\) & \(2\) & \(1\) & \(1\) \\
\(\beta\)  & \(1\) & \(2\) & \(2\) & \(0\) \\
\hline
\(\gamma\) & \(0\) & \(1\) & \(1\) & \(0\) \\
\hline
\end{tabular}
\end{center}
\end{table}


\end{theorem}

\begin{proof}
A part of the proof is due to Barrucand and Cohn
\cite{BaCo}
who distinguished \(4\) different types,
\(\mathrm{I}\hat{=}\beta\), \(\mathrm{II}\), \(\mathrm{III}\hat{=}\alpha\), and \(\mathrm{IV}\hat{=}\gamma\).
However, Halter-Koch
\cite{HK}
showed the impossibility of one of these types, namely type \(\mathrm{II}\).
Our new proof with different methods is given in
\cite[Thm. 6.2]{AMITA}.
\end{proof}


\subsection{Differential principal factorization (DPF) types of real dihedral fields}
\label{ss:QuinticTypes}
\noindent
Now we state the classification theorem
for \textit{pure quintic} fields \(L=\mathbb{Q}(\sqrt[5]{D})\)
and their Galois closure \(N=\mathbb{Q}(\zeta_5,\sqrt[5]{D})\),
that is the metacyclic case \(p=5\).
The \textit{coarse} classification of \(N\)
according to the invariants \(U\) and \(A\) alone
is closely related to the
classification of \textit{totally real dihedral} fields of degree \(2p\) with any odd prime \(p\)
by Nicole Moser
\cite[Thm. III.5, p. 62]{Mo},
as illustrated in Figure
\ref{fig:MoserExtendedQuintic}.
The coarse types \(\alpha\), \(\beta\), \(\gamma\), \(\delta\), \(\varepsilon\)
are completely analogous in both cases.
Additional types \(\zeta\), \(\eta\), \(\vartheta\) are required for pure quintic fields,
because there arises the possibility that the primitive fifth root of unity \(\zeta_5\)
occurs as relative norm \(N_{N/K}(Z)\) of a unit \(Z\in U_N\).
Due to the existence of radicals in the pure quintic case,
the \(\mathbb{F}_p\)-dimension \(A\) of the vector space of absolute DPF
exceeds the corresponding dimension for totally real dihedral fields by one
(see Remark
\ref{rmk:Estimates}).


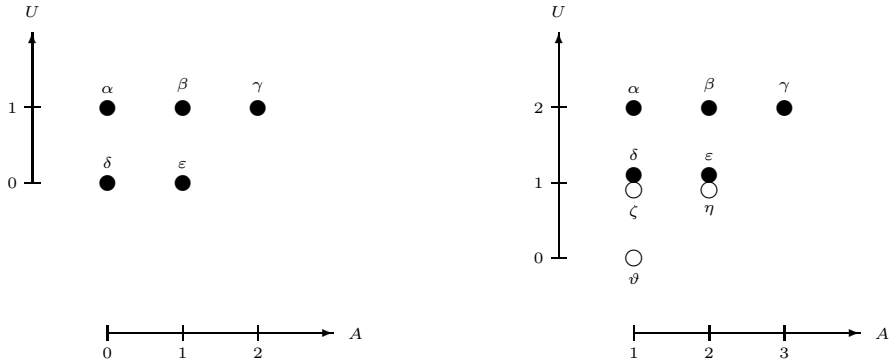
\begin{figure}[ht]
\caption{Classification of Totally Real Dihedral and Pure Quintic Fields}
\label{fig:MoserExtendedQuintic}

{\tiny

\setlength{\unitlength}{1.0cm}
\begin{picture}(15,5)(-11,-10)




\put(-9,-5.8){\makebox(0,0)[cb]{\(U\)}}
\put(-9,-7){\vector(0,1){1}}
\put(-9,-8){\line(0,1){1}}
\multiput(-9.1,-8)(0,1){2}{\line(1,0){0.2}}

\put(-9.2,-7){\makebox(0,0)[rc]{\(1\)}}
\put(-9.2,-8){\makebox(0,0)[rc]{\(0\)}}

\put(-4.8,-10){\makebox(0,0)[lc]{\(A\)}}
\put(-6,-10){\vector(1,0){1}}
\put(-8,-10){\line(1,0){2}}
\multiput(-8,-10.1)(1,0){3}{\line(0,1){0.2}}

\put(-8,-10.2){\makebox(0,0)[ct]{\(0\)}}
\put(-7,-10.2){\makebox(0,0)[ct]{\(1\)}}
\put(-6,-10.2){\makebox(0,0)[ct]{\(2\)}}


\put(-8,-7){\circle*{0.2}}
\put(-8,-6.8){\makebox(0,0)[cb]{\(\alpha\)}}
\put(-7,-7){\circle*{0.2}}
\put(-7,-6.8){\makebox(0,0)[cb]{\(\beta\)}}
\put(-6,-7){\circle*{0.2}}
\put(-6,-6.8){\makebox(0,0)[cb]{\(\gamma\)}}

\put(-8,-8){\circle*{0.2}}
\put(-8,-7.8){\makebox(0,0)[cb]{\(\delta\)}}
\put(-7,-8){\circle*{0.2}}
\put(-7,-7.8){\makebox(0,0)[cb]{\(\varepsilon\)}}




\put(-2,-5.8){\makebox(0,0)[cb]{\(U\)}}
\put(-2,-7){\vector(0,1){1}}
\put(-2,-9){\line(0,1){2}}
\multiput(-2.1,-9)(0,1){3}{\line(1,0){0.2}}

\put(-2.2,-7){\makebox(0,0)[rc]{\(2\)}}
\put(-2.2,-8){\makebox(0,0)[rc]{\(1\)}}
\put(-2.2,-9){\makebox(0,0)[rc]{\(0\)}}

\put(2.2,-10){\makebox(0,0)[lc]{\(A\)}}
\put(1,-10){\vector(1,0){1}}
\put(-1,-10){\line(1,0){2}}
\multiput(-1,-10.1)(1,0){3}{\line(0,1){0.2}}

\put(-1,-10.2){\makebox(0,0)[ct]{\(1\)}}
\put(0,-10.2){\makebox(0,0)[ct]{\(2\)}}
\put(1,-10.2){\makebox(0,0)[ct]{\(3\)}}


\put(-1,-7){\circle*{0.2}}
\put(-1,-6.8){\makebox(0,0)[cb]{\(\alpha\)}}
\put(0,-7){\circle*{0.2}}
\put(0,-6.8){\makebox(0,0)[cb]{\(\beta\)}}
\put(1,-7){\circle*{0.2}}
\put(1,-6.8){\makebox(0,0)[cb]{\(\gamma\)}}

\put(-1,-7.9){\circle*{0.2}}
\put(-1,-7.7){\makebox(0,0)[cb]{\(\delta\)}}
\put(0,-7.9){\circle*{0.2}}
\put(0,-7.7){\makebox(0,0)[cb]{\(\varepsilon\)}}
\put(-1,-8.1){\circle{0.2}}
\put(-1,-8.3){\makebox(0,0)[ct]{\(\zeta\)}}
\put(0,-8.1){\circle{0.2}}
\put(0,-8.3){\makebox(0,0)[ct]{\(\eta\)}}

\put(-1,-9){\circle{0.2}}
\put(-1,-9.2){\makebox(0,0)[ct]{\(\vartheta\)}}


\end{picture}
}
\end{figure}


\noindent
The \textit{fine} classification of \(N\) according to the invariants \(U\), \(A\), \(R\) and \(C\)
in the totally real dihedral situation with \(U+1=A+R+C\)
splits type \(\alpha\) with \(U=1\), \(A=0\) further in
type \(\alpha_1\) with \(C=2\) (double capitulation),
type \(\alpha_2\) with \(C=R=1\) (mixed capitulation and relative DPF),
type \(\alpha_3\) with \(R=2\) (double relative DPF),
type \(\beta\) with \(U=A=1\) in
type \(\beta_1\) with \(C=1\) (capitulation),
type \(\beta_2\) with \(R=1\) (relative DPF), and
type \(\delta\) with \(U=A=0\) in
type \(\delta_1\) with \(C=1\) (capitulation),
type \(\delta_2\) with \(R=1\) (relative DPF).
In the pure quintic situation with \(U+1=A+I+R\)
\cite{Ma2018},
however, we arrive at the second of the following theorems
where we oppose the two classifications.

\begin{theorem}
\label{thm:MainReal}
Each totally real dihedral field \(N/\mathbb{Q}\)
of absolute degree \(\lbrack N:\mathbb{Q}\rbrack=2p\) with an odd prime \(p\)
belongs to precisely one of the following \(9\) differential principal factorization types,
in dependence on the invariant \(U\) and the triplet \((A,R,C)\).


\renewcommand{\arraystretch}{1.1}

\begin{table}[ht]
\label{tbl:RealDPFTypes}
\begin{center}
\begin{tabular}{|c||c||c||ccc|}
\hline
 Type           & \(U\) & \(U+1=A+R+C\) & \(A\) & \(R\) & \(C\) \\
\hline
\(\alpha_1\)    & \(1\) & \(2\) & \(0\) & \(0\) & \(2\) \\
\(\alpha_2\)    & \(1\) & \(2\) & \(0\) & \(1\) & \(1\) \\
\(\alpha_3\)    & \(1\) & \(2\) & \(0\) & \(2\) & \(0\) \\
\(\beta_1\)     & \(1\) & \(2\) & \(1\) & \(0\) & \(1\) \\
\(\beta_2\)     & \(1\) & \(2\) & \(1\) & \(1\) & \(0\) \\
\(\gamma\)      & \(1\) & \(2\) & \(2\) & \(0\) & \(0\) \\
\hline
\(\delta_1\)    & \(0\) & \(1\) & \(0\) & \(0\) & \(1\) \\
\(\delta_2\)    & \(0\) & \(1\) & \(0\) & \(1\) & \(0\) \\
\(\varepsilon\) & \(0\) & \(1\) & \(1\) & \(0\) & \(0\) \\
\hline
\end{tabular}
\end{center}
\end{table}


\end{theorem}

\begin{proof}
Consequence of the Corollaries
\ref{cor:Trichotomy}
and 
\ref{cor:Estimates}.
See also \cite[Thm. III.5, p. 62]{Mo} and \cite{Ma1991}.
\end{proof}


\begin{theorem}
\label{thm:MainQuintic}
Each pure metacyclic field \(N=\mathbb{Q}(\zeta_5,\sqrt[5]{D})\)
of absolute degree \(\lbrack N:\mathbb{Q}\rbrack=20\)
with \(5\)-th power free radicand \(D\in\mathbb{Z}\), \(D\ge 2\),
belongs to precisely one of the following \(13\) differential principal factorization types,
in dependence on the invariant \(U\) and the triplet \((A,I,R)\).


\renewcommand{\arraystretch}{1.1}

\begin{table}[ht]
\label{tbl:QuinticDPFTypes}
\begin{center}
\begin{tabular}{|c||c||c||ccc|}
\hline
 Type           & \(U\) & \(U+1=A+I+R\) & \(A\) & \(I\) & \(R\) \\
\hline
\(\alpha_1\)    & \(2\) & \(3\) & \(1\) & \(0\) & \(2\) \\
\(\alpha_2\)    & \(2\) & \(3\) & \(1\) & \(1\) & \(1\) \\
\(\alpha_3\)    & \(2\) & \(3\) & \(1\) & \(2\) & \(0\) \\
\(\beta_1\)     & \(2\) & \(3\) & \(2\) & \(0\) & \(1\) \\
\(\beta_2\)     & \(2\) & \(3\) & \(2\) & \(1\) & \(0\) \\
\(\gamma\)      & \(2\) & \(3\) & \(3\) & \(0\) & \(0\) \\
\hline
\(\delta_1\)    & \(1\) & \(2\) & \(1\) & \(0\) & \(1\) \\
\(\delta_2\)    & \(1\) & \(2\) & \(1\) & \(1\) & \(0\) \\
\(\varepsilon\) & \(1\) & \(2\) & \(2\) & \(0\) & \(0\) \\
 \hline
\(\zeta_1\)     & \(1\) & \(2\) & \(1\) & \(0\) & \(1\) \\
\(\zeta_2\)     & \(1\) & \(2\) & \(1\) & \(1\) & \(0\) \\
\(\eta\)        & \(1\) & \(2\) & \(2\) & \(0\) & \(0\) \\
 \hline
\(\vartheta\)   & \(0\) & \(1\) & \(1\) & \(0\) & \(0\) \\
\hline
\end{tabular}
\end{center}
\end{table}


The types \(\delta_1\), \(\delta_2\), \(\varepsilon\)
are characterized additionally by \(\zeta_5\not\in N_{N/K}(U_N)\),
and the types \(\zeta_1\), \(\zeta_2\), \(\eta\)
by \(\zeta_5\in N_{N/K}(U_N)\).
\end{theorem}

\begin{proof}
The proof is given in
\cite[Thm. 6.1]{Ma2018}.
\end{proof}

\newpage

\section{Complete verification of the conjecture of Scholz}
\label{s:ScholzConjecture}
\noindent
Let \(L\) be a non-cyclic \textit{totally real cubic field}.
Then \(L\) is non-Galois over the rational number field \(\mathbb{Q}\)
with two conjugate fields \(L^\prime\) and \(L^{\prime\prime}\).
The Galois closure \(N\) of \(L\) is a totally real dihedral field of degree \(6\),
which contains a unique real quadratic field \(K\),
as illustrated in Figure
\ref{fig:AbsoluteSubfields}.


\begin{figure}[h]
\caption{Hasse Subfield Diagram of \(N/\mathbb{Q}\)}
\label{fig:AbsoluteSubfields}

{\small

\setlength{\unitlength}{1.0cm}
\begin{picture}(5,5)(-7,-9.4)



\put(-6,-9){\circle*{0.2}}
\put(-6,-9.2){\makebox(0,0)[ct]{\(\mathbb{Q}=L\cap K\)}}
\put(-7,-9){\makebox(0,0)[rc]{rational number field}}

\put(-6,-9){\line(2,1){2}}
\put(-5,-8.7){\makebox(0,0)[lt]{\(\lbrack K:\mathbb{Q}\rbrack=2\)}}

\put(-4,-8){\circle*{0.2}}
\put(-4,-8.2){\makebox(0,0)[ct]{\(K\)}}
\put(-3,-8){\makebox(0,0)[lc]{quadratic field}}


\put(-6.2,-7.5){\makebox(0,0)[rc]{\(\lbrack L:\mathbb{Q}\rbrack=3\)}}
\put(-6,-9){\line(0,1){3}}
\put(-4,-8){\line(0,1){3}}



\put(-6,-6){\circle{0.2}}
\put(-6,-5.8){\makebox(0,0)[cb]{\(L\)}}
\put(-5.8,-6){\makebox(0,0)[lt]{\(L^\prime,L^{\prime\prime}\)}}
\put(-7,-6){\makebox(0,0)[rc]{three conjugate cubic fields}}

\put(-6,-6){\line(2,1){2}}

\put(-4,-5){\circle*{0.2}}
\put(-4,-4.8){\makebox(0,0)[cb]{\(N=L\cdot K\)}}
\put(-3,-5){\makebox(0,0)[lc]{dihedral field of degree \(6\)}}


\end{picture}

}

\end{figure}
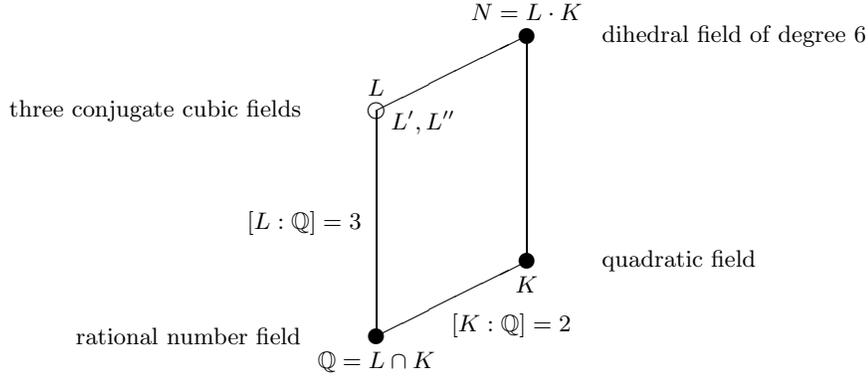


In \(1930\), Helmut Hasse
\cite{Ha1930}
determined the \textit{discriminants} of \(L\)
\cite[pp. 567 and 575]{Ha1930}
and \(N\)
\cite[p. 566]{Ha1930},
in dependence on the discriminant of \(K\)
and on the class field theoretic \textit{conductor} \(f=f(N/K)\)
of the cyclic cubic, and thus abelian, relative extension \(N/K\):
\begin{equation}
\label{eqn:Hasse}
d_L=f^2\cdot d_K, \quad \text{ and } \quad d_N=f^4\cdot d_K^3.
\end{equation}

Three years later, in \(1933\), Arnold Scholz
\cite[p. 216]{So}
determined the \textit{relation}
\begin{equation}
\label{eqn:Scholz}
h_N=\frac{I}{9}\cdot h_L^2\cdot h_K
\end{equation}
\textit{between the class numbers} of the fields \(N\), \(L\) and \(K\),
in dependence on the \textit{index of subfield units}, \(I=(U_N:U_0)=3^E\),
where \(U_0=\langle U_K,U_L,U_{L^\prime},U_{L^{\prime\prime}}\rangle\)
and \(E\in\lbrace 0,1,2\rbrace\).

Note that \(E=0\), resp. \(I=1\), is the \textit{distinguished case}
where the unit group \(U_N\) of the normal field \(N\)
is entirely generated by all proper subfield units, that is \(U_N=U_0\).


Scholz was able to give explicit numerical examples
\cite[p. 216]{So},
for \(E=1\) (\(d_L=229\)),
resp. \(E=2\) (\(d_L=148\)),
but not for \(E=0\),
and he formulated the following hypothesis.

\begin{conjecture}
\label{cnj:Scholz}
(The Conjecture of Scholz, \(1933\), illustrated in Figure
\ref{fig:RingClassFields}) \\
There should exist non-Galois totally real cubic fields \(L\)
whose Galois closure \(N\) is either 
\begin{enumerate}
\item
\textit{unramified}, with conductor \(f=1\),
over some real quadratic field \(K\) with \(3\)-class rank \(\varrho_3(K)=2\)
whose complete \(3\)-elementary class group capitulates in \(N\),
and such that \(U_N=U_0\)
\cite[p. 219]{So},
or
\item
\textit{ramified}, with conductor \(f>1\),
over some real quadratic field \(K\)
such that also \(U_N=U_0\)
\cite[p. 221]{So}
(here, Scholz calls \(N\) a \textit{ring class field} over \(K\), by abuse of language).
\end{enumerate}
\end{conjecture}


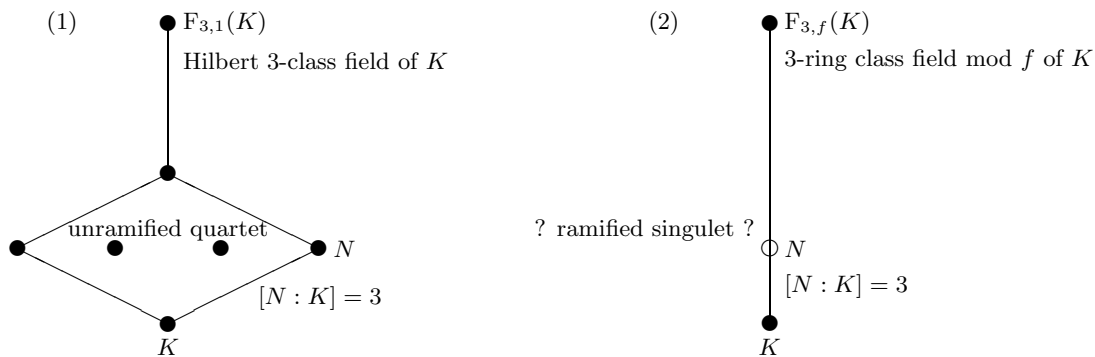
\begin{figure}[h]
\caption{Hilbert and Ring Class Fields over \(K\)}
\label{fig:RingClassFields}

{\small

\setlength{\unitlength}{1.0cm}
\begin{picture}(10,6)(-8,-10)



\put(-8,-9){\circle*{0.2}}
\put(-8,-9.2){\makebox(0,0)[ct]{\(K\)}}

\put(-8,-9){\line(-2,1){2}}
\put(-8,-9){\line(2,1){2}}
\put(-6.8,-8.5){\makebox(0,0)[lt]{\(\lbrack N:K\rbrack=3\)}}

\put(-10,-8){\line(2,1){2}}
\put(-6,-8){\line(-2,1){2}}

\put(-8,-7.9){\makebox(0,0)[cb]{unramified quartet}}
\put(-10,-8){\circle*{0.2}}
\put(-8.7,-8){\circle*{0.2}}
\put(-7.3,-8){\circle*{0.2}}
\put(-6,-8){\circle*{0.2}}
\put(-5.8,-8){\makebox(0,0)[lc]{\(N\)}}

\put(-8,-7){\circle*{0.2}}

\put(-8,-7){\line(0,1){2}}

\put(-9.2,-5){\makebox(0,0)[rc]{(1)}}
\put(-8,-5){\circle*{0.2}}
\put(-7.8,-5){\makebox(0,0)[lc]{\(\mathrm{F}_{3,1}(K)\)}}
\put(-7.8,-5.5){\makebox(0,0)[lc]{Hilbert \(3\)-class field of \(K\)}}



\put(0,-9){\circle*{0.2}}
\put(0,-9.2){\makebox(0,0)[ct]{\(K\)}}

\put(0,-9){\line(0,1){1}}
\put(0.2,-8.5){\makebox(0,0)[lc]{\(\lbrack N:K\rbrack=3\)}}
\put(0,-8){\line(0,1){3}}

\put(-0.2,-7.9){\makebox(0,0)[rb]{? ramified singulet ?}}
\put(0,-8){\circle{0.2}}
\put(0.2,-8){\makebox(0,0)[lc]{\(N\)}}

\put(-1.2,-5){\makebox(0,0)[rc]{(2)}}
\put(0,-5){\circle*{0.2}}
\put(0.2,-5){\makebox(0,0)[lc]{\(\mathrm{F}_{3,f}(K)\)}}
\put(0.2,-5.5){\makebox(0,0)[lc]{\(3\)-ring class field mod \(f\) of \(K\)}}


\end{picture}

}

\end{figure}


\noindent
We point out that, in the \textit{unramified} situation \(f=1\),
\(d_L=d_K\) is a quadratic fundamental discriminant, and
\(d_N=d_K^3\) is a perfect cube,
according to formula
\eqref{eqn:Hasse}.
In this unramified case,
the verification of Conjecture
\ref{cnj:Scholz}
can be obtained from a more general theorem, since
any real quadratic field \(K\) with \(3\)-class rank \(\varrho_3(K)=2\)
possesses a multiplet of four unramified cyclic cubic extensions
\(N_1,\ldots,N_4\),
that is a \textit{quartet} of absolutely dihedral fields of degree \(6\)
\cite{Ma2012}
with non-Galois totally real subfields \(L_1,\ldots,L_4\),
each of them selected among three conjugate fields.


For such a quartet,
Chang and Foote
\cite{ChFt}
introduced the concept of the \textit{capitulation number} \(0\le\nu(K)\le 4\),
defined as the number of those members of the quartet
in which the complete \(3\)-elementary class group of \(K\) capitulates.
For this number \(\nu(K)\), the following theorem holds.

\begin{theorem}
\label{thm:Alpha1Unramified}
For each value \(0\le\nu\le 4\),
there exists a real quadratic field \(K\) with \(3\)-class rank \(\varrho_3(K)=2\)
and capitulation number \(\nu(K)=\nu\).
It is even possible to restrict the claim to fields with
elementary \(3\)-class group of type \(\mathrm{Cl}_3(K)\simeq C_3\times C_3\).
\end{theorem}

\newpage

\begin{proof}
From the viewpoint of finite \(p\)-group theory,
this theorem is a proven statement about the possible \textit{transfer kernel types}
of finite metabelian \(3\)-groups \(G\) with
abelianization \(G/G^\prime\simeq (3,3)\)
applied to the second \(3\)-class group \(G:=\mathrm{Gal}(F_3^2(K)/K)\) of \(K\)
\cite{Ma2012}.
However, it is easier to give explicit numerical paradigms for each value of \(\nu(K)\).
We have the following minimal occurrences: \\
\(\nu(K)=4\) for \(d_K=62\,501\), \\
\(\nu(K)=3\) for \(d_K=32\,009\), \\
\(\nu(K)=2\) for \(d_K=710\,652\), \\
\(\nu(K)=1\) for \(d_K=534\,824\), \\
\(\nu(K)=0\) for \(d_K=214\,712\), \\
which have been computed by ourselves in
\cite{Ma2012}.
This completes the proof of Theorem
\ref{thm:Alpha1Unramified}.
\end{proof}


\begin{remark}
\label{rmk:Alpha1Unramified}
We have the priority of discovering the first examples of
real quadratic fields \(K\) with \(\nu(K)\in\lbrace 0,1,2\rbrace\) in
\cite{Ma2012}.
However, the first examples of
real quadratic fields \(K\) with \(\nu(K)\in\lbrace 3,4\rbrace\)
are due to Heider and Schmithals
\cite{HeSm},
who performed a mainframe computation on the CDC Cyber of the University of Cologne,
and thus the following corollary is proven since \(1982\) already.
\end{remark}


\begin{corollary}
\label{cor:Alpha1Unramified}
(Verification of Conjecture \ref{cnj:Scholz}, (1) for \textbf{unramified} extensions; see Figure
\ref{fig:HilbertClassFieldQuartet}) \\
There exist non-Galois totally real cubic fields \(L\)
whose Galois closure \(N\) is unramified, with conductor \(f=1\),
over a real quadratic field \(K\) with \(3\)-class rank \(\varrho_3(K)=2\)
whose complete \(3\)-elementary class group capitulates in \(N\),
and which therefore has \(U_N=U_0\).
The minimal discriminant of such a field \(L\) is \(d_L=32\,009\)
(three of four fields, \cite{HeSm}, \(49\) years after \cite{So}).
\end{corollary}


\begin{figure}[h]
\caption{Hilbert Class Field over \(K\)}
\label{fig:HilbertClassFieldQuartet}

{\small

\setlength{\unitlength}{1.0cm}
\begin{picture}(10,5)(-12,-9.6)



\put(-8,-9){\circle*{0.2}}
\put(-8,-9.2){\makebox(0,0)[ct]{\(K\)}}

\put(-8,-9){\line(-2,1){2}}
\put(-8,-9){\line(2,1){2}}
\put(-6.8,-8.5){\makebox(0,0)[lt]{\(\lbrack N:K\rbrack=3\)}}

\put(-10,-8){\line(2,1){2}}
\put(-6,-8){\line(-2,1){2}}

\put(-8,-7.9){\makebox(0,0)[cb]{unramified quartet}}
\put(-10,-8){\circle*{0.2}}
\put(-8.7,-8){\circle*{0.2}}
\put(-7.3,-8){\circle*{0.2}}
\put(-6,-8){\circle*{0.2}}
\put(-5.8,-8){\makebox(0,0)[lc]{\(N\)}}

\put(-8,-7){\circle*{0.2}}

\put(-8,-7){\line(0,1){2}}

\put(-8,-5){\circle*{0.2}}
\put(-7.8,-5){\makebox(0,0)[lc]{\(\mathrm{F}_{3,1}(K)\)}}
\put(-7.8,-5.5){\makebox(0,0)[lc]{Hilbert \(3\)-class field of \(K\)}}


\end{picture}

}

\end{figure}
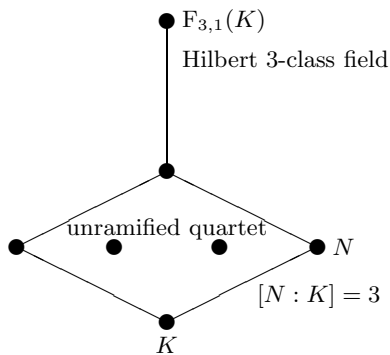


\begin{proof}
It suffices to take a real quadratic field \(K\) with \(1\le \nu(K)\le 4\)
in Theorem
\ref{thm:Alpha1Unramified}.
In view of the minimal discriminant, we select \(\nu(K)=3\)
and obtain \(U_N=U_0\) for \(d_L=d_K=32\,009\).
\end{proof}


\noindent
Concerning the \textit{ramified} situation  \(f>1\) in Conjecture
\ref{cnj:Scholz} (2),
Scholz does not explicitly impose any conditions
on the underlying real quadratic field \(K\).
We suppose that he also tacitly assumed
a real quadratic field \(K\) with \(3\)-class rank \(\varrho_3(K)=2\).
However, more recent extensions of the theory of dihedral fields
by means of \textit{differential principal factorizations} and \textit{Galois cohomology},
two concepts which we have expanded thoroughly in the preparatory sections \S\S\
\ref{ss:NormKernel},
\ref{ss:CapitulationKernel}, and
\ref{ss:GaloisCohomology},
revealed that for \(U_N=U_0\)
no constraints on the \(p\)-class rank \(\varrho_p(K)\) are required.
In \(1975\), Nicole Moser
\cite{Mo}
used the \textit{Galois cohomology}
\(\mathrm{H}^0(G,U_N)\simeq U_K/\mathrm{N}_{N/K}(U_N)\)
of the unit group \(U_N\) of the normal closure \(N\) as a module over \(G=\mathrm{Gal}(N/K)\)
to establish a \textit{fine structure} with five possible types \(\alpha,\beta,\gamma,\delta,\varepsilon\)
on the \textit{coarse} classification by three possible values of the index of subfield units: \\
\((U_N:U_0)=1\) \(\Longleftrightarrow\) type \(\alpha\) with \((U_K:\mathrm{N}_{N/K}(U_N))=3\), \\
\((U_N:U_0)=3\) \(\Longleftrightarrow\) type \(\beta\) with \((U_K:\mathrm{N}_{N/K}(U_N))=3\) or type \(\delta\) with \((U_K:\mathrm{N}_{N/K}(U_N))=1\), \\
\((U_N:U_0)=9\) \(\Longleftrightarrow\) type \(\gamma\) with \((U_K:\mathrm{N}_{N/K}(U_N))=3\) or type \(\varepsilon\) with \((U_K:\mathrm{N}_{N/K}(U_N))=1\). \\
Thus, Moser's refinement does not illuminate the situation \(U_N=U_0\) (\(\Longleftrightarrow\) type \(\alpha\)) of Scholz's conjecture more closely.
Meanwhile, Barrucand and Cohn
\cite{BaCo}
had coined the concept of \textit{differential principal factorization} (DPF) for pure cubic fields.
In \(1991\), we generalized the theory of DPFs for dihedral fields of both signatures
\cite{Ma1991},
and we obtained a \textit{hyperfine structure} by splitting Moser's types further
according to the \(\mathbb{F}_p\)-dimensions
\(C\) of the capitulation kernel \(\ker(T_{K,N})\) and
\(R\) of the space of relative DPFs of \(N/K\),
which we recalled in the preparatory section \S\
\ref{ss:QuinticTypes}.
In particular, type \(\alpha\) with \(U_N=U_0\) splits into three subtypes: \\
type \(\alpha_1\) \(\Longleftrightarrow\) \(C=2\), \(R=0\), which implies \(\varrho_p(K)\ge 2\), \\
type \(\alpha_2\) \(\Longleftrightarrow\) \(C=1\), \(R=1\), which implies \(\varrho_p(K)\ge 1\) and a split prime divisor of \(f\) \((s\ge 1)\), \\
type \(\alpha_3\) \(\Longleftrightarrow\) \(C=0\), \(R=2\), which is compatible with any \(\varrho_p(K)\ge 0\), but requires \(s\ge 2\).


Consequently, we were led to the following refinement of Conjecture
\ref{cnj:Scholz} (2).

\begin{conjecture}
\label{cnj:Mayer}
(Conjecture of D. C. Mayer, \(1991\)) \\
Non-Galois totally real cubic fields \(L\)
whose Galois closure \(N\) is \textit{ramified}, with conductor \(f>1\),
over some real quadratic field \(K\),
and is of type \(\alpha\), with \(U_N=U_0\),
should exist for each of the following three situations:
\begin{enumerate}
\item[(2.1)]
type \(\alpha_1\) with
\(\dim_{\mathbb{F}_3}(\ker(T_{K,N}))=2\) and \(\varrho_3(K)=2\), \(s=0\),
\item[(2.2)]
type \(\alpha_2\) with
\(\dim_{\mathbb{F}_3}(\ker(T_{K,N}))=1\) and \(\varrho_3(K)=1\), \(s=1\),
\item[(2.3)]
type \(\alpha_3\) with
\(\dim_{\mathbb{F}_3}(\ker(T_{K,N}))=0\) and \(\varrho_3(K)=0\), \(s=2\),
\end{enumerate}
where \(T_{K,N}:\,\mathrm{Cl}_3(K)\to\mathrm{Cl}_3(N)\), \(\mathfrak{a}\cdot\mathcal{P}_K\mapsto(\mathfrak{a}\mathcal{O}_N)\cdot\mathcal{P}_N\),
denotes the \textit{transfer homomorphism} of \(3\)-classes from \(K\) to \(N\),
and \(s\) counts the prime divisors of the conductor \(f\) which \textit{split} in \(K\).
\end{conjecture}


\begin{figure}[ht]
\caption{Ring Class Field modulo \(f=63=3^2\cdot 7\) over \(K\)}
\label{fig:RingClassFieldSingulet}

{\small

\setlength{\unitlength}{0.9cm}
\begin{picture}(10,4.5)(-4,-9.2)



\put(0,-9){\circle*{0.2}}
\put(0,-9.2){\makebox(0,0)[ct]{\(K\)}}

\put(0,-9){\line(0,1){1}}
\put(0.2,-8.5){\makebox(0,0)[lc]{\(\lbrack N:K\rbrack=3\)}}
\put(0,-8){\line(0,1){3}}

\put(-0.2,-7.9){\makebox(0,0)[rb]{ramified singulet}}
\put(0,-8){\circle{0.2}}
\put(0.2,-8){\makebox(0,0)[lc]{\(N\)}}

\put(0,-5){\circle*{0.2}}
\put(0.2,-5){\makebox(0,0)[lc]{\(\mathrm{F}_{3,f}(K)\)}}
\put(0.2,-5.5){\makebox(0,0)[lc]{\(3\)-ring class field mod \(f\) of \(K\)}}


\end{picture}

}

\end{figure}
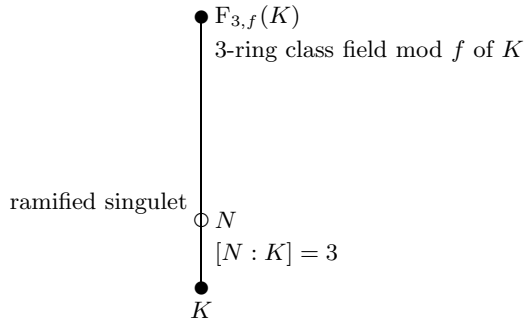


\begin{theorem}
\label{thm:Alpha3Ramified}
(Verification of Conjecture \ref{cnj:Mayer}, (2.3), and Conjecture \ref{cnj:Scholz}, (2); see Figure
\ref{fig:RingClassFieldSingulet}) \\
There exist non-Galois totally real cubic fields \(L\)
whose Galois closure \(N\) is ramified,
with conductor \(f>1\) divisible by two prime divisors which split in \(K\),
over a real quadratic field \(K\) with \(3\)-class rank \(\varrho_3(K)=0\),
without capitulation in \(N\),
but which nevertheless has \(U_N=U_0\).
The minimal discriminant of such a field \(L\) is \(d_L=146\,853=(7\cdot 9)^2\cdot 37\)
(singulet, \cite{Ma}, \(58\) years after \cite{So}).
\end{theorem}

\begin{proof}
This was proved in the numerical supplement
\cite{Ma}
of our paper
\cite{Ma1991}
by computing a gapless list of all \(10\,015\) totally real cubic fields \(L\)
with discriminants \(d_L<200\,000\)
on the AMDAHL mainframe of the University of Manitoba.
There occurred the minimal discriminant \(d_L=146\,853=f^2\cdot d_K\)
with \(d_K=37\) and conductor \(f=63=3^2\cdot 7\)
divisible by two primes which both split in \(K\), i.e. \(s=2\).
This is a necessary requirement for a two-dimensional relative principal factorization with \(R=2\)
and is unique up to \(d_L<200\,000\).
(The next is \(d_L=240\,149\) with \(f=7\cdot 13\).)
There is only a single field \(L\) with this discriminant \(d_L=146\,853\) (forming a singulet).
\end{proof}


Our discovery of the truth of Theorem
\ref{thm:Alpha3Ramified}
with the aid of the list
\cite{Ma}
was a random hit without explicit intention to find a verification of Scholz's conjecture.
Unfortunately,
\cite{Ma}
does not contain examples of the unique missing DPF type \(\alpha_2\).
It required more than \(25\) years until we focused on an attack against this lack of information.
In contrast to the techniques of
\cite{Ma},
we did not use the Voronoi algorithm
\cite{Vo}
after cumbersome preparation of generating polynomials for totally real cubic fields,
but rather the class field theory routines of Magma
\cite{BCP,BCFS,MAGMA}
for a direct generation of the fields as subfields of \(3\)-ray class fields modulo conductors \(f>1\).


\begin{theorem}
\label{thm:Alpha2Ramified}
(Verification of Conjecture \ref{cnj:Mayer}, (2.2), and Conjecture \ref{cnj:Scholz}, (2); see Figure
\ref{fig:RingClassFieldQuartet}) \\
There exist non-Galois totally real cubic fields \(L\)
whose Galois closure \(N\) is ramified,
with conductor \(f>1\) divisible by a single prime divisor that splits in \(K\), i.e. \(s=1\),
over a real quadratic field \(K\) with \(3\)-class rank \(\varrho_3(K)=1\),
with complete capitulation of the elementary \(3\)-class group in \(N\),
but nevertheless with \(U_N=U_0\).
The minimal discriminant of such a field \(L\) is \(d_L=966\,397=19^2\cdot 2\,677\)
(two of three fields, 19 November \(2017\), \(84\) years after \cite{So}, \(1933\)).
\end{theorem}

\begin{proof}
The proof is conducted in the following section \S\
\ref{ss:Rank1}.
\end{proof}

\begin{figure}[ht]
\caption{Heterogeneous Quartet modulo \(f=19\) over \(K\)}
\label{fig:RingClassFieldQuartet}

{\small

\setlength{\unitlength}{1.0cm}
\begin{picture}(10,6)(-8,-9)



\put(-5,-9){\circle*{0.2}}
\put(-5,-9.2){\makebox(0,0)[ct]{\(K\)}}

\put(-5,-9){\line(-1,1){1}}
\put(-5,-9){\line(1,1){1}}
\put(-5,-9){\line(5,1){5}}

\put(-4.1,-7.3){\makebox(0,0)[cb]{heterogeneous quartet}}
\put(-6.2,-7.8){\makebox(0,0)[rb]{unramified singulet}}
\put(-6,-8){\circle*{0.2}}

\put(-6,-8){\line(0,1){2}}

\put(-9.2,-4){\makebox(0,0)[rc]{(2.2)}}
\put(-6,-6){\circle*{0.2}}
\put(-6,-5.7){\makebox(0,0)[cb]{\(\mathrm{F}_{3,1}(K)\)}}
\put(-6.2,-6.2){\makebox(0,0)[rc]{Hilbert \(3\)-class field of \(K\)}}



\put(-4.2,-8.4){\makebox(0,0)[lc]{\(\lbrack N:K\rbrack=3\)}}
\put(-2,-4){\line(-2,-1){4}}
\put(-2,-4){\line(-1,-2){2}}
\put(-2,-4){\line(1,-2){2}}

\put(-2,-7.8){\makebox(0,0)[cb]{ramified triplet}}
\put(-4,-8){\circle{0.2}}
\put(-2,-8){\circle{0.2}}
\put(0,-8){\circle{0.2}}
\put(-4.2,-8){\makebox(0,0)[rc]{\(N\)}}

\put(-2,-4){\circle*{0.2}}
\put(-2,-3.8){\makebox(0,0)[cb]{\(\mathrm{F}_{3,f}(K)\)}}
\put(-1.8,-4.1){\makebox(0,0)[lc]{\(3\)-ring class field mod \(f\) of \(K\)}}


\end{picture}

}

\end{figure}
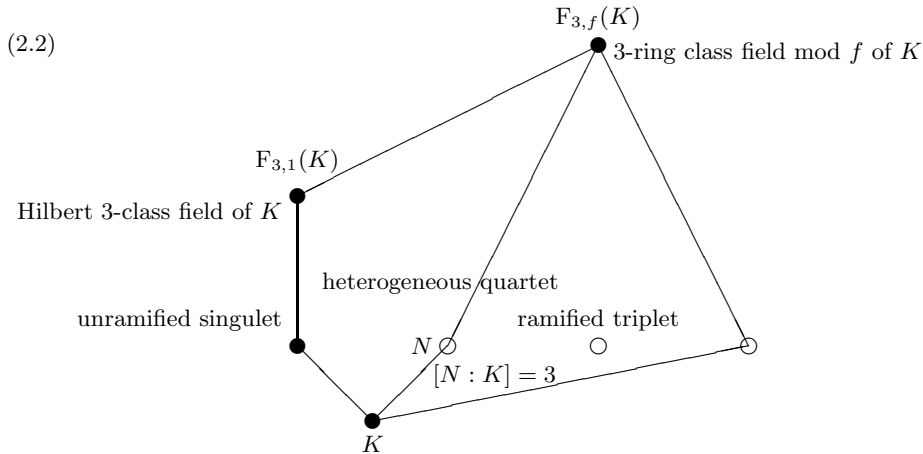


\begin{theorem}
\label{thm:Alpha1Ramified}
(Verification of Conjecture \ref{cnj:Mayer}, (2.1), and Conjecture \ref{cnj:Scholz}, (2); see Figure
\ref{fig:RingClassFieldTridecuplet}) \\
There exist non-Galois totally real cubic fields \(L\)
whose Galois closure \(N\) is ramified,
with conductor \(f>1\) divisible only by prime divisors which do not split in \(K\),
over a real quadratic field \(K\) with \(3\)-class rank \(\varrho_3(K)=2\),
with complete capitulation of the elementary \(3\)-class group in \(N\),
and thus with \(U_N=U_0\).
The minimal discriminant of such a field \(L\) is \(d_L=18\,251\,060=2^2\cdot 4\,562\,765\) \\
(five of nine fields, 23 November \(2017\), \(84\) years after \cite{So}, \(1933\)).
\end{theorem}

\begin{figure}[ht]
\caption{Heterogeneous Tridecuplet modulo \(f=2\) over \(K\)}
\label{fig:RingClassFieldTridecuplet}

{\small

\setlength{\unitlength}{1.0cm}
\begin{picture}(10,7)(-8,-9)



\put(-5,-9){\circle*{0.2}}
\put(-5,-9.2){\makebox(0,0)[ct]{\(K\)}}

\put(-5,-9){\line(-5,1){5}}
\put(-5,-9){\line(-1,1){1}}
\put(-5,-9){\line(1,1){1}}
\put(-5,-9){\line(5,1){5}}

\put(-10,-8){\line(2,1){2}}
\put(-6,-8){\line(-2,1){2}}

\put(-5.5,-7.3){\makebox(0,0)[cb]{heterogeneous tridecuplet}}
\put(-8,-7.9){\makebox(0,0)[cb]{unramified quartet}}
\put(-10,-8){\circle*{0.2}}
\put(-8.7,-8){\circle*{0.2}}
\put(-7.3,-8){\circle*{0.2}}
\put(-6,-8){\circle*{0.2}}

\put(-8,-7){\circle*{0.2}}

\put(-8,-7){\line(0,1){2}}

\put(-9.2,-3){\makebox(0,0)[rc]{(2.1)}}
\put(-8,-5){\circle*{0.2}}
\put(-8,-4.8){\makebox(0,0)[cb]{\(\mathrm{F}_{3,1}(K)\)}}
\put(-7.8,-5.2){\makebox(0,0)[lc]{Hilbert \(3\)-class field of \(K\)}}



\put(-4.2,-8.4){\makebox(0,0)[lc]{\(\lbrack N:K\rbrack=3\)}}
\put(0,-3){\line(-4,-1){8}}
\put(0,-3){\line(-4,-5){4}}
\put(0,-3){\line(0,-1){5}}

\put(-2,-7.8){\makebox(0,0)[cb]{ramified nonet}}
\put(-4,-8){\circle{0.2}}
\put(-3.5,-8){\circle{0.2}}
\put(-3,-8){\circle{0.2}}
\put(-2.5,-8){\circle{0.2}}
\put(-2,-8){\circle{0.2}}
\put(-1.5,-8){\circle{0.2}}
\put(-1,-8){\circle{0.2}}
\put(-0.5,-8){\circle{0.2}}
\put(0,-8){\circle{0.2}}
\put(-4.2,-8){\makebox(0,0)[rc]{\(N\)}}

\put(0,-3){\circle*{0.2}}
\put(0,-2.8){\makebox(0,0)[cb]{\(\mathrm{F}_{3,f}(K)\)}}
\put(0.2,-3.1){\makebox(0,0)[lc]{\(3\)-ring class field mod \(f\) of \(K\)}}


\end{picture}

}

\end{figure}
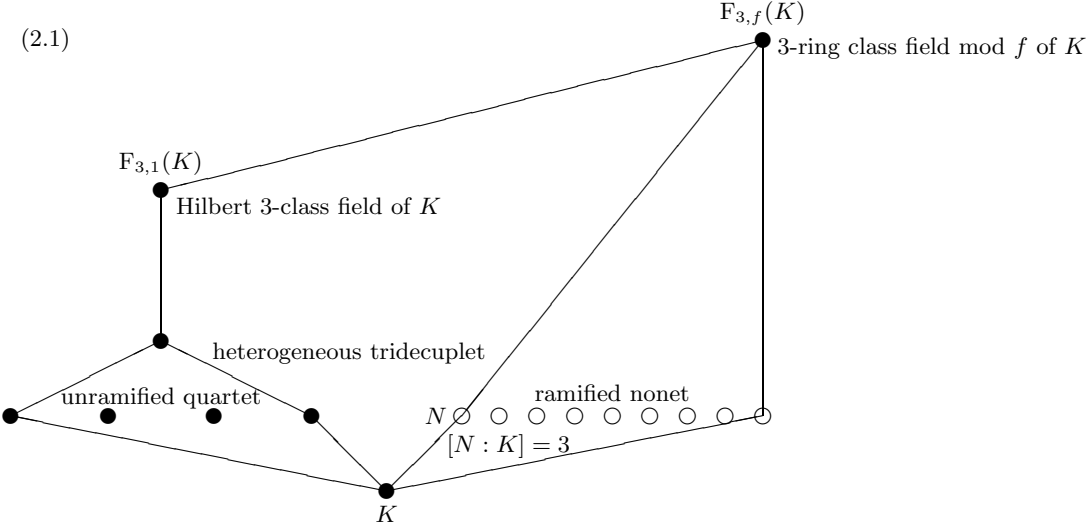

\begin{proof}
The proof is conducted in the following section \S\
\ref{ss:Rank2}.
\end{proof}

The proof of Theorem
\ref{thm:Alpha1Ramified}
and Theorem
\ref{thm:Alpha2Ramified}
is conducted in the following sections
on real quadratic base fields with \(3\)-class rank \(1\) and \(2\).

\subsection{Real quadratic base fields with \(3\)-class rank \(1\)}
\label{ss:Rank1}

\noindent
In Table
\ref{tbl:Rank1},
we present the results of our search
for the \textit{minimal discriminant} \(d_L\), resp. \(d_N\),
of a non-Galois totally real cubic field \(L\), resp. its normal closure \(N\),
with \textit{differential principal factorization type} \(\alpha_2\).
The unramified component is a \textit{singulet} which must be of DPF type \(\delta_1\).
For each member of the ramified \textit{triplet}
the DPF types \(\alpha_2,\beta_1,\beta_2,\delta_1,\delta_2,\varepsilon\) would be possible,
but only the types \(\alpha_2,\delta_1,\delta_2\) occur usually.

The desired minimum is clearly given by \(d_L=19^2\cdot 2\,677=966\,397\)
with two occurrences of ramified extensions with DPF type \(\alpha_2\).
For \(f=3^2\), the condition \(d_K\equiv 1\,(\mathrm{mod}\,3)\) is required.

\renewcommand{\arraystretch}{1.1}

\begin{table}[ht]
\caption{Heterogeneous Quartets of Dihedral Fields with a Splitting Prime \(f\)}
\label{tbl:Rank1}
\begin{center}
\begin{tabular}{|r|r|r||c||ccc|}
\hline
        &             &                      & unramified component & \multicolumn{3}{c|}{ramified components}   \\
 \(f\)  & \(d_K\)     & \(d_L=f^2\cdot d_K\) & \(\delta_1\)         & \(\alpha_2\) & \(\delta_1\) & \(\delta_2\) \\
\hline
\(3^2\) & \(14\,197\) &  \(1\,149\,957\)     & \(1\)                & \(3\)        & \(0\)        & \(0\)        \\
 \(7\)  & \(21\,781\) &  \(1\,067\,269\)     & \(1\)                & \(2\)        & \(1\)        & \(0\)        \\
 \(13\) &  \(9\,749\) &  \(1\,647\,581\)     & \(1\)                & \(2\)        & \(0\)        & \(1\)        \\
 \(19\) &  \(2\,677\) &     \(966\,397\)     & \(1\)                & \(2\)        & \(1\)        & \(0\)        \\
 \(31\) &  \(3\,877\) &  \(3\,725\,797\)     & \(1\)                & \(2\)        & \(0\)        & \(1\)        \\
 \(37\) &  \(5\,477\) &  \(7\,498\,013\)     & \(1\)                & \(1\)        & \(0\)        & \(2\)        \\
 \(43\) &  \(4\,933\) &  \(9\,121\,117\)     & \(1\)                & \(3\)        & \(0\)        & \(0\)        \\
 \(61\) &  \(3\,981\) & \(14\,813\,301\)     & \(1\)                & \(3\)        & \(0\)        & \(0\)        \\
 \(67\) &  \(4\,493\) & \(20\,169\,077\)     & \(1\)                & \(2\)        & \(0\)        & \(1\)        \\
 \(73\) & \(10\,733\) & \(57\,196\,157\)     & \(1\)                & \(3\)        & \(0\)        & \(0\)        \\
\hline
\end{tabular}
\end{center}
\end{table}


\noindent
Since we know a small candidate \(d_L=966\,397\) for the minimal discriminant,
and since the smallest quadratic discriminant with \(\varrho_3(K)=1\) is \(d_K=229\),
we only have to investigate prime and composite conductors \(f=\sqrt{\frac{d_L}{d_K}}\) with \(s\ge 1\) and
\[f\le\sqrt{\frac{966\,397}{229}}\approx\sqrt{4220}\approx 64.9,\]
which are divisible by a split prime, that is,
\[f\in\lbrace 7,9=3^2,13,14=2\cdot 7,18=2\cdot 3^2,19,21=3\cdot 7,26=2\cdot 13,31,35=5\cdot 7,37,\]
\[38=2\cdot 19,39=3\cdot 13,42=2\cdot 3\cdot 7,43,45=5\cdot 3^2,57=3\cdot 19,61,62=2\cdot 31,63=7\cdot 3^2\rbrace.\]

{\normalsize

\renewcommand{\arraystretch}{1.1}

\begin{table}[ht]
\caption{Heterogeneous Quartets of Dihedral Fields with Conductor \(f\)}
\label{tbl:Rank1Coarse}
\begin{center}
\begin{tabular}{|r|c||r|r||c||ccccc|}
\hline
                    &                      &             &                      & unramified component & \multicolumn{5}{c|}{ramified components}                               \\
 \(f\)              & condition            & \(d_K\)     & \(d_L=f^2\cdot d_K\) & \(\delta_1\)         & \(\alpha_2\) & \(\beta_1\) & \(\beta_2\) & \(\delta_1\) & \(\delta_2\) \\
\hline
              \(7\) &                      & \(21\,781\) &      \(1\,067\,269\) & \(1\)                & \(2\)        & \(0\)       & \(0\)       & \(1\)        & \(0\)        \\
            \(3^2\) & \(d_K\equiv 1\,(3)\) & \(14\,197\) &      \(1\,149\,957\) & \(1\)                & \(3\)        & \(0\)       & \(0\)       & \(0\)        & \(0\)        \\
             \(13\) &                      &  \(9\,749\) &      \(1\,647\,581\) & \(1\)                & \(2\)        & \(0\)       & \(0\)       & \(0\)        & \(1\)        \\
             \(19\) &                      &  \(2\,677\) &         \(966\,397\) & \(1\)                & \(2\)        & \(0\)       & \(0\)       & \(1\)        & \(0\)        \\
             \(31\) &                      &  \(3\,877\) &      \(3\,725\,797\) & \(1\)                & \(2\)        & \(0\)       & \(0\)       & \(0\)        & \(1\)        \\
             \(37\) &                      &  \(5\,477\) &      \(7\,498\,013\) & \(1\)                & \(1\)        & \(0\)       & \(0\)       & \(0\)        & \(2\)        \\
             \(43\) &                      &  \(4\,933\) &      \(9\,121\,117\) & \(1\)                & \(3\)        & \(0\)       & \(0\)       & \(0\)        & \(0\)        \\
             \(61\) &                      &  \(3\,981\) &     \(14\,813\,301\) & \(1\)                & \(3\)        & \(0\)       & \(0\)       & \(0\)        & \(0\)        \\
\hline
       \(2\cdot 7\) &                      &  \(6\,997\) &      \(1\,371\,412\) & \(1\)                & \(3\)        & \(0\)       & \(0\)       & \(0\)        & \(0\)        \\
     \(2\cdot 3^2\) & \(d_K\equiv 1\,(3)\) & \(16\,141\) &      \(5\,229\,684\) & \(1\)                & \(3\)        & \(0\)       & \(0\)       & \(0\)        & \(0\)        \\
       \(3\cdot 7\) & \(d_K\equiv 3\,(9)\) & \(28\,137\) &     \(12\,408\,417\) & \(1\)                & \(3\)        & \(0\)       & \(0\)       & \(0\)        & \(0\)        \\
       \(3\cdot 7\) & \(d_K\equiv 6\,(9)\) & \(57\,516\) &     \(25\,364\,556\) & \(1\)                & \(3\)        & \(0\)       & \(0\)       & \(0\)        & \(0\)        \\
      \(2\cdot 13\) &                      & \(21\,557\) &     \(14\,572\,532\) & \(1\)                & \(3\)        & \(0\)       & \(0\)       & \(0\)        & \(0\)        \\
       \(5\cdot 7\) &                      & \(14\,457\) &     \(17\,709\,825\) & \(1\)                & \(3\)        & \(0\)       & \(0\)       & \(0\)        & \(0\)        \\
      \(2\cdot 19\) &                      & \(13\,765\) &     \(19\,876\,660\) & \(1\)                & \(3\)        & \(0\)       & \(0\)       & \(0\)        & \(0\)        \\
      \(3\cdot 13\) & \(d_K\equiv 3\,(9)\) & \(51\,528\) &     \(78\,374\,088\) & \(1\)                & \(3\)        & \(0\)       & \(0\)       & \(0\)        & \(0\)        \\
      \(3\cdot 13\) & \(d_K\equiv 6\,(9)\) & \(37\,176\) &     \(56\,544\,696\) & \(1\)                & \(3\)        & \(0\)       & \(0\)       & \(0\)        & \(0\)        \\
\(2\cdot 3\cdot 7\) & \(d_K\equiv 3\,(9)\) &\(891\,237\) & \(1\,572\,142\,068\) & \(1\)                & \(4\)        & \(1\)       & \(1\)       & \(0\)        & \(0\)        \\
\(2\cdot 3\cdot 7\) & \(d_K\equiv 6\,(9)\) &\(474\,261\) &    \(836\,596\,404\) & \(1\)                & \(2\)        & \(0\)       & \(1\)       & \(0\)        & \(0\)        \\
     \(5\cdot 3^2\) & \(d_K\equiv 1\,(3)\) & \(24\,952\) &     \(50\,527\,800\) & \(1\)                & \(1\)        & \(0\)       & \(0\)       & \(1\)        & \(1\)        \\
      \(3\cdot 19\) & \(d_K\equiv 3\,(9)\) & \(24\,393\) &     \(79\,252\,857\) & \(1\)                & \(3\)        & \(0\)       & \(0\)       & \(0\)        & \(0\)        \\
      \(3\cdot 19\) & \(d_K\equiv 6\,(9)\) & \(39\,417\) &    \(128\,065\,833\) & \(1\)                & \(3\)        & \(0\)       & \(0\)       & \(0\)        & \(0\)        \\
      \(2\cdot 31\) &                      &  \(7\,573\) &     \(29\,110\,612\) & \(1\)                & \(3\)        & \(0\)       & \(0\)       & \(0\)        & \(0\)        \\
     \(7\cdot 3^2\) & \(d_K\equiv 1\,(3)\) &  \(2\,941\) &     \(11\,672\,829\) & \(1\)                & \(3\)        & \(0\)       & \(0\)       & \(0\)        & \(0\)        \\
     \(7\cdot 3^2\) & \(d_K\equiv 2\,(3)\) & \(23\,993\) &     \(95\,228\,217\) & \(1\)                & \(3\)        & \(0\)       & \(0\)       & \(0\)        & \(0\)        \\
\hline
\end{tabular}
\end{center}
\end{table}

}

\noindent
The result of the investigations is summarized in Table
\ref{tbl:Rank1Coarse},
which clearly shows that \(d_L=\mathbf{966\,397}\), for \(d_K=2\,677\) and splitting prime conductor \(f=19\)
bigger than the conductor \(f=1\) of unramified extensions \(N/K\),
is the desired \textbf{minimal discriminant} of a totally real cubic field with
ramified extension \(N/K\), DPF type \(\alpha_2\) and \(U_N=U_0\).
The information has been computed with class field theoretic routines of Magma
\cite{MAGMA}.


\subsection{Real quadratic base fields with \(3\)-class rank \(2\)}
\label{ss:Rank2}
\noindent
In this situation, the unramified \textit{quartet} is non-trivial,
since the two DPF types \(\alpha_1\) and \(\delta_1\) are possible.
These quartets have been studied thoroughly in
\cite{Ma2012},
and in the Tables
\ref{tbl:Rank2Cond2}
and
\ref{tbl:Rank2Cond5},
we use the corresponding notation for \textit{capitulation types}.

In Table
\ref{tbl:Rank2Cond2},
we present the results of the crucial search
for the \textit{minimal discriminant} \(d_L\), resp. \(d_N\),
of a non-Galois totally real cubic field \(L\), resp. its normal closure \(N\),
with \textit{differential principal factorization type} \(\alpha_1\)
such that \(N/K\) is a \textit{ramified} extension of a
real quadratic field \(K\) with \(3\)-class rank \(\varrho_3=2\).
We tried to fix the minimal possible conductor \(f>1\), namely \(f=2\).
This experiment was motivated by the fact that
the conductor \(f\) enters the expression \(d_L=f^2\cdot d_K\) in its second power,
whereas the quadratic discriminant \(d_K\) enters linearly.
Consequently, the probability to find the minimum of \(d_L\)
is higher for small \(f\) than for small \(d_K\).

The table is ordered by increasing quadratic fundamental discriminants \(d_K\)
and gives \(d_L=2^2\cdot d_K\) and the \textit{extended Artin pattern}
of the \textit{heterogeneous tridecuplet} of cyclic cubic relative extensions \(N/K\)
consisting of an \textit{unramified quartet} \((N_{1,1},\ldots,N_{1,4})\) with conductor \(f^\prime=1\)
and a \textit{ramified nonet} \((N_{2,1},\ldots,N_{2,9})\) with conductor \(f=2\),
grouped by the possible two, resp. four, DPF types.
Capitulation kernels \(\varkappa\) are abbreviated by digits,
\(0\) for two-dimensional and
\(1,\ldots,4\) for one-dimensional principalization,
and an asterisk \(\ast\) for a trivial kernel.
Transfer targets \(\tau\) are abbreviated
by logarithmic abelian type invariants of \(3\)-class groups.
Symbolic exponents always denote repetition.

The desired minimum is given by \(d_L=4\cdot 4\,562\,765=18\,251\,060\)
with five occurrences of ramified extensions with DPF type \(\alpha_1\).
Generally,
there is an abundance of ramified extensions with two-dimensional capitulation kernel:
at least three and at most all nine of a nonet.


\renewcommand{\arraystretch}{1.1}

\begin{table}[ht]
\caption{Heterogeneous multiplets of Artin patterns for \(f=2\)}
\label{tbl:Rank2Cond2}
\begin{center}
\begin{tabular}{|r||l|cc|cc||cc|cc|cc|cc|}
\hline
                 & \multicolumn{5}{c||}{unramified components} & \multicolumn{8}{c|}{ramified components} \\
                 &      & \multicolumn{2}{c|}{\(\alpha_1\)} & \multicolumn{2}{c||}{\(\delta_1\)} & \multicolumn{2}{c|}{\(\alpha_1\)} & \multicolumn{2}{c|}{\(\beta_1\)} & \multicolumn{2}{c|}{\(\delta_1\)} & \multicolumn{2}{c|}{\(\varepsilon\)} \\
         \(d_K\) & Type & \(\varkappa\) & \(\tau\)          & \(\varkappa\) & \(\tau\)           & \(\varkappa\) & \(\tau\)          & \(\varkappa\) & \(\tau\)         & \(\varkappa\) & \(\tau\)          & \(\varkappa\) & \(\tau\)             \\
\hline
 \(4\,562\,765\) & a.\(3^\ast\) & \(0^3\) & \((1^2)^3\) & \(1\)    & \(1^3\)        & \(0^5\) & \(2^21^2,(1^4)^4\) & \(1\) & \(1^5\)  & \(14^2\)    & \((21^3)^3\)     &       &          \\
 \(7\,339\,397\) & a.\(3^\ast\) & \(0^3\) & \((1^2)^3\) & \(1\)    & \(1^3\)        & \(0^7\) & \((1^4)^7\)        & \(2\) & \(21^3\) & \(1\)       & \(21^3\)         &       &          \\
 \(7\,601\,461\) & a.\(3\)      & \(0^3\) & \((1^2)^3\) & \(1\)    & \(21\)         & \(0^6\) & \(2^21^2,(1^4)^5\) &       &          & \(234\)     & \((21^3)^3\)     &       &          \\
 \(7\,657\,037\) & a.\(3\)      & \(0^3\) & \((1^2)^3\) & \(1\)    & \(21\)         & \(0^6\) & \((1^4)^6\)        & \(1\) & \(21^3\) & \(12\)      & \(1^5,21^3\)     &       &          \\
 \(7\,736\,749\) & a.\(3^\ast\) & \(0^3\) & \((1^2)^3\) & \(1\)    & \(1^3\)        & \(0^7\) & \((1^4)^7\)        &       &          & \(4^2\)     & \((21^3)^2\)     &       &          \\
 \(8\,102\,053\) & a.\(3^\ast\) & \(0^3\) & \((1^2)^3\) & \(1\)    & \(1^3\)        & \(0^7\) & \((1^4)^7\)        &       &          & \(23\)      & \(1^5,21^3\)     &       &          \\
 \(9\,182\,229\) & a.\(2\)      & \(0^3\) & \((1^2)^3\) & \(4\)    & \(21\)         & \(0^8\) & \(2^21^2,(1^4)^7\) &       &          & \(2\)       & \(21^3\)         &       &          \\
 \(9\,500\,453\) & a.\(3\)      & \(0^3\) & \((1^2)^3\) & \(1\)    & \(21\)         & \(0^8\) & \(2^21^2,(1^4)^7\) &       &          & \(3\)       & \(21^3\)         &       &          \\
 \(9\,533\,357\) & a.\(3\)      & \(0^3\) & \((1^2)^3\) & \(1\)    & \(21\)         & \(0^6\) & \((1^4)^6\)        & \(1\) & \(21^3\) & \(23\)      & \((21^3)^2\)     &       &          \\
\(11\,003\,845\) & a.\(3\)      & \(0^3\) & \((1^2)^3\) & \(1\)    & \(21\)         & \(0^4\) & \((1^4)^4\)        &       &          & \(12^24^2\) & \(1^5,(21^3)^4\) &       &          \\
\(12\,071\,253\) & a.\(3\)      & \(0^3\) & \((1^2)^3\) & \(1\)    & \(21\)         & \(0^7\) & \((1^4)^7\)        & \(3\) & \(21^3\) & \(2\)       & \(21^3\)         &       &          \\
\(14\,266\,853\) & a.\(3\)      & \(0^3\) & \((1^2)^3\) & \(1\)    & \(21\)         & \(0^8\) & \(2^21^2,(1^4)^7\) &       &          & \(4\)       & \(21^3\)         &       &          \\
\(14\,308\,421\) & a.\(3^\ast\) & \(0^3\) & \((1^2)^3\) & \(1\)    & \(1^3\)        & \(0^4\) & \((1^4)^4\)        &       &          & \(1^2234\)  & \(2^31,(21^3)^3,1^5\) &  &          \\
\(14\,315\,765\) & a.\(3\)      & \(0^3\) & \((1^2)^3\) & \(1\)    & \(21\)         & \(0^7\) & \((1^4)^7\)        &       &          & \(23\)      & \((21^3)^2\)     &       &          \\
\(14\,395\,013\) & a.\(3^\ast\) & \(0^3\) & \((1^2)^3\) & \(1\)    & \(1^3\)        & \(0^6\) & \((1^4)^6\)        & \(1\) & \(21^3\) & \(23\)      & \((21^3)^2\)     &       &          \\
\(15\,131\,149\) & D.\(10\)     &         &             & \(2414\) & \((21)^3,1^3\) & \(0^7\) & \((1^4)^7\)        & \(1\) & \(21^3\) & \(1\)       & \(21^3\)         &       &          \\
\(16\,385\,741\) & a.\(3^\ast\) & \(0^3\) & \((1^2)^3\) & \(1\)    & \(1^3\)        & \(0^4\) & \((1^4)^4\)        &       &          & \(23^24\)   & \((21^3)^4\)     & \(\ast\) & \(32^21\) \\
\hline
\end{tabular}
\end{center}
\end{table}


\noindent
Table
\ref{tbl:Rank2Cond5}
shows analogous results for the conductor \(f=5\),
that is, \(d_L=5^2\cdot d_K\).
The minimum \(d_L=25\cdot 1\,049\,512=26\,237\,800\)
is clearly beaten by the minimum \(4\cdot 4\,562\,765=18\,251\,060\) in Table
\ref{tbl:Rank2Cond2}.


\renewcommand{\arraystretch}{1.1}

\begin{table}[ht]
\caption{Heterogeneous multiplets of Artin patterns for \(f=5\)}
\label{tbl:Rank2Cond5}
\begin{center}
\begin{tabular}{|r||l|cc|cc||cc|cc|cc|cc|}
\hline
                 & \multicolumn{5}{c||}{unramified components} & \multicolumn{8}{c|}{ramified components} \\
                 &      & \multicolumn{2}{c|}{\(\alpha_1\)} & \multicolumn{2}{c||}{\(\delta_1\)} & \multicolumn{2}{c|}{\(\alpha_1\)} & \multicolumn{2}{c|}{\(\beta_1\)} & \multicolumn{2}{c|}{\(\delta_1\)} & \multicolumn{2}{c|}{\(\varepsilon\)} \\
         \(d_K\) & Type & \(\varkappa\) & \(\tau\)          & \(\varkappa\) & \(\tau\)           & \(\varkappa\) & \(\tau\)          & \(\varkappa\) & \(\tau\)         & \(\varkappa\) & \(\tau\)          & \(\varkappa\) & \(\tau\)             \\
\hline
 \(1\,049\,512\) & a.\(3\)      & \(0^3\) & \((1^2)^3\) & \(1\)    & \(21\)         & \(0^4\) & \((1^4)^4\)            &       &          & \(234^3\) & \((21^3)^5\) &       &          \\
 \(2\,461\,537\) & a.\(2\)      & \(0^3\) & \((1^2)^3\) & \(4\)    & \(21\)         & \(0^7\) & \((1^4)^7\)            &       &          & \(12\)    & \((21^3)^2\) &       &          \\
 \(2\,811\,613\) & a.\(3^\ast\) & \(0^3\) & \((1^2)^3\) & \(1\)    & \(1^3\)        & \(0^5\) & \(2^21^2,(1^4)^4\)     & \(2\) & \(21^3\) & \(123\)   & \((21^3)^3\) &       &          \\
 \(3\,091\,133\) & a.\(3\)      & \(0^3\) & \((1^2)^3\) & \(1\)    & \(21\)         & \(0^4\) & \((1^4)^4\)            & \(4\) & \(21^3\) & \(1^32\)  & \((21^3)^4\) &       &          \\
 \(5\,858\,753\) & G.\(19\)     &         &             & \(2143\) & \((21)^4\)     & \(0^7\) & \((2^21^2)^3,(1^4)^4\) &       &          & \(3\)     & \(21^3\)     & \(\ast\) & \(21^4\)  \\
 \(6\,036\,188\) & D.\(10\)     &         &             & \(3431\) & \(1^3,(21)^3\) & \(0^8\) & \((1^4)^8\)            &       &          &           &              & \(\ast\) & \(2^21^2\) \\
\hline
\end{tabular}
\end{center}
\end{table}


\noindent
Since we know a small candidate \(d_L=18\,251\,060\) for the minimal discriminant,
and since the smallest quadratic discriminant with \(\varrho_3(K)=2\) is \(d_K=32\,009\),
we only have to investigate prime and composite conductors \(f=\sqrt{\frac{d_L}{d_K}}\) with
\[f\le\sqrt{\frac{18\,251\,060}{32\,009}}\approx\sqrt{570.2}\approx 23.9,\]
that is,
\[f\in\lbrace 2,3,5,6=2\cdot 3,7,9=3^2,10=2\cdot 5,11,13,14=2\cdot 7,\]
\[15=3\cdot 5,17,18=2\cdot 3^2,19,21=3\cdot 7,22=2\cdot 11,23\rbrace.\]

{\normalsize

\renewcommand{\arraystretch}{1.1}

\begin{table}[ht]
\caption{Heterogeneous tridecuplets of dihedral fields with conductor \(f\)}
\label{tbl:Rank2Coarse}
\begin{center}
\begin{tabular}{|r|c||r|r||cc||cccc|}
\hline
               &                      &                 &                      & \multicolumn{2}{c||}{unramified components} & \multicolumn{4}{c|}{ramified components} \\
 \(f\)         & condition            & \(d_K\)         & \(d_L=f^2\cdot d_K\) & \(\alpha_1\) & \(\delta_1\) & \(\alpha_1\) & \(\beta_1\) & \(\delta_1\) & \(\varepsilon\) \\
\hline
         \(2\) &                      & \(4\,562\,765\) &     \(18\,251\,060\) & \(3\)        & \(1\)        & \(5\)        & \(1\)       & \(3\)        & \(0\)           \\
         \(3\) & \(d_K\equiv 3\,(9)\) & \(9\,964\,821\) &     \(89\,683\,389\) & \(3\)        & \(1\)        & \(4\)        & \(0\)       & \(4\)        & \(1\)           \\
         \(5\) &                      & \(1\,049\,512\) &     \(26\,237\,800\) & \(3\)        & \(1\)        & \(4\)        & \(0\)       & \(5\)        & \(0\)           \\
         \(7\) &                      &    \(966\,053\) &     \(47\,336\,597\) & \(3\)        & \(1\)        & \(4\)        & \(0\)       & \(4\)        & \(1\)           \\
       \(3^2\) & \(d_K\equiv 1\,(3)\) & \(1\,482\,568\) &    \(120\,088\,008\) & \(3\)        & \(1\)        & \(5\)        & \(1\)       & \(2\)        & \(1\)           \\
       \(3^2\) & \(d_K\equiv 2\,(3)\) & \(2\,515\,388\) &    \(203\,746\,428\) & \(3\)        & \(1\)        & \(6\)        & \(1\)       & \(2\)        & \(0\)           \\
       \(3^2\) & \(d_K\equiv 6\,(9)\) &    \(621\,429\) &     \(50\,335\,749\) & \(3\)        & \(1\)        & \(6\)        & \(0\)       & \(3\)        & \(0\)           \\
        \(11\) &                      &    \(476\,152\) &     \(57\,614\,392\) & \(3\)        & \(1\)        & \(7\)        & \(0\)       & \(2\)        & \(0\)           \\
        \(13\) &                      & \(1\,122\,573\) &    \(189\,714\,837\) & \(3\)        & \(1\)        & \(7\)        & \(0\)       & \(2\)        & \(0\)           \\
        \(17\) &                      &    \(665\,832\) &    \(192\,425\,848\) & \(3\)        & \(1\)        & \(7\)        & \(0\)       & \(2\)        & \(0\)           \\
        \(19\) &                      &    \(635\,909\) &    \(229\,563\,149\) & \(3\)        & \(1\)        & \(5\)        & \(3\)       & \(1\)        & \(0\)           \\
        \(23\) &                      &    \(390\,876\) &    \(206\,773\,404\) & \(3\)        & \(1\)        & \(7\)        & \(1\)       & \(1\)        & \(0\)           \\
\hline
  \(2\cdot 3\) & \(d_K\equiv 3\,(9)\) & \(5\,963\,493\) &    \(214\,685\,748\) & \(3\)        & \(1\)        & \(7\)        & \(2\)       & \(0\)        & \(0\)           \\
  \(2\cdot 3\) & \(d_K\equiv 6\,(9)\) & \(4\,305\,957\) &    \(155\,014\,452\) & \(0\)        & \(4\)        & \(6\)        & \(3\)       & \(0\)        & \(0\)           \\
  \(2\cdot 5\) &                      &    \(363\,397\) &     \(36\,339\,700\) & \(3\)        & \(1\)        & \(6\)        & \(3\)       & \(0\)        & \(0\)           \\
  \(2\cdot 7\) &                      &    \(358\,285\) &     \(70\,223\,860\) & \(4\)        & \(0\)        & \(7\)        & \(2\)       & \(0\)        & \(0\)           \\
  \(3\cdot 5\) & \(d_K\equiv 3\,(9)\) & \(4\,845\,432\) & \(1\,090\,222\,200\) & \(3\)        & \(1\)        & \(6\)        & \(3\)       & \(0\)        & \(0\)           \\
  \(3\cdot 5\) & \(d_K\equiv 6\,(9)\) & \(1\,646\,817\) &    \(370\,533\,825\) & \(3\)        & \(1\)        & \(6\)        & \(3\)       & \(0\)        & \(0\)           \\
\(2\cdot 3^2\) & \(d_K\equiv 1\,(3)\) & \(2\,142\,445\) &    \(694\,152\,180\) & \(3\)        & \(1\)        & \(6\)        & \(3\)       & \(0\)        & \(0\)           \\
\(2\cdot 3^2\) & \(d_K\equiv 2\,(3)\) &    \(635\,909\) &    \(206\,034\,516\) & \(3\)        & \(1\)        & \(6\)        & \(3\)       & \(0\)        & \(0\)           \\
\(2\cdot 3^2\) & \(d_K\equiv 6\,(9)\) & \(2\,538\,285\) &    \(822\,404\,340\) & \(3\)        & \(1\)        & \(6\)        & \(3\)       & \(0\)        & \(0\)           \\
  \(3\cdot 7\) & \(d_K\equiv 3\,(9)\) & \(3\,597\,960\) & \(1\,586\,700\,360\) & \(3\)        & \(1\)        & \(6\)        & \(3\)       & \(0\)        & \(0\)           \\
  \(3\cdot 7\) & \(d_K\equiv 6\,(9)\) & \(3\,122\,232\) & \(1\,376\,904\,312\) & \(0\)        & \(4\)        & \(6\)        & \(3\)       & \(0\)        & \(0\)           \\
 \(2\cdot 11\) &                      & \(2\,706\,373\) & \(1\,309\,884\,532\) & \(3\)        & \(1\)        & \(6\)        & \(3\)       & \(0\)        & \(0\)           \\
\hline
\end{tabular}
\end{center}
\end{table}

}

\noindent
The result of the investigations is summarized in Table
\ref{tbl:Rank2Coarse},
which clearly shows that \(d_L=\mathbf{18\,251\,060}\), for \(d_K=4\,562\,765\) and the smallest possible conductor \(f=2\)
bigger than the conductor \(f=1\) of unramified extensions \(N/K\),
is the desired \textbf{minimal discriminant} of a totally real cubic field with
ramified extension \(N/K\), DPF type \(\alpha_1\) and \(U_N=U_0\).
The information has been computed with class field theoretic routines of Magma
\cite{MAGMA}.


\section{Acknowledgements}
\label{s:Thanks}

\noindent
We gratefully acknowledge that our research was supported by the Austrian Science Fund (FWF):
projects J 0497-PHY and P 26008-N25.



\end{document}